\documentclass[11pt]{amsart}
\usepackage[all]{xy}

\usepackage{latexsym}
\usepackage{amsmath}
\usepackage{amsfonts}
\usepackage{amssymb}
\usepackage{amsthm}
\usepackage{eucal}
\usepackage{enumerate,yfonts}
\usepackage{mathrsfs}
\usepackage{graphicx}
\usepackage{graphics}
\usepackage{epstopdf}
\usepackage{amscd}
\usepackage{bbm}
\usepackage{hyperref}
\usepackage{url}

\usepackage{bbm}
\usepackage{cancel}
\usepackage{enumerate}
\usepackage{amsmath,amsthm}
\usepackage{amssymb}
\usepackage{epsfig}
\usepackage{pstricks}
\usepackage{xy}

\DeclareSymbolFont{AMSb}{U}{msb}{m}{n}
\DeclareMathSymbol{\N}{\mathbin}{AMSb}{"4E}
\DeclareMathSymbol{\Z}{\mathbin}{AMSb}{"5A}
\DeclareMathSymbol{\Q}{\mathbin}{AMSb}{"51}
\DeclareMathSymbol{\I}{\mathbin}{AMSb}{"49}
\DeclareMathSymbol{\F}{\mathbin}{AMSb}{"46}

\newcommand{\ZZ}{\mathfrak{Z}_{crit}}

\newcommand{\OO}{\Omega^1(\R)}
\newcommand{\DD}{\mathcal{D}^{0}_{crit}}
\newcommand{\DDC}{\mathcal{D}_{crit}}
\newcommand{\DDH}{\mathcal{D}_{\hbar}}

\newcommand{\LAH}{\mathcal{A}_{\hbar}}
\newcommand{\LAHH}{\mathcal{A}_{-\hbar}}
\newcommand{\AL}{\mathcal{A}_{crit}}
\newcommand{\gc}{\hat{\mathfrak{g}}_{crit}}

\newcommand{\WW}{\mathcal{W}_{\hbar}}
\newcommand{\M}{\mathcal{M}}
\newcommand{\B}{\mathcal{B}^0}
\newcommand{\INDZ}{\text{Ind}_{\ZZ}(\ZZ^c)}
\newcommand{\INDZC}{\text{Ind}_{\ZZ}^{ch}(\ZZ^c)}
\newcommand{\INDL}{\text{Ind}_{\R}(L)}
\newcommand{\A}{\mathcal{A}}
\newcommand{\C}{\mathcal{C}}

\newcommand*{\longhookrightarrow}{\ensuremath{\lhook\joinrel\relbar\joinrel\rightarrow}}

\theoremstyle{plain}
\newtheorem{thm}{Theorem}[section]
\newtheorem{lem}[thm]{Lemma}
\newtheorem{prop}[thm]{Proposition}

\newtheorem{defprop}[thm]{Definition-Proposition}

\theoremstyle{definition}
\newtheorem{defn}{Definition}[section]
\newtheorem{remark}[thm]{Remark}

\newcommand{\bigslant}[2]{{\raisebox{.2em}{$#1$}\left/\raisebox{-.2em}{$#2$}\right.}}
\newcommand{\IND}{\text{Ind}_{\R}(\RC)}
\newcommand{\INDR}{\text{Ind}_{\R}^{ch}(\RC)}
\newcommand{\g}{\mathfrak{g}}

\DeclareMathOperator{\End}{End}

\DeclareMathOperator{\Hom}{Hom}

\renewcommand{\tilde}{\widetilde}

\newcommand{\DX}{\mathcal{D}_{X}}
\newcommand{\R}{\mathcal{R}}
\newcommand{\OC}{\Omega^{c}(\R)}

\newcommand{\RC}{\mathcal{R}^{c}}

\newcommand{\LI}{\mathcal{L}}

\newcommand{\INDC}{\text{Ind}_{\R}^{cl}(L^c)}
\newcommand{\INDLD}{\text{Ind}_{\R}^{ch}(L^c)}
\newcommand{\INDLL}{\text{Ind}_{\R}(L^c)}

\newcommand{\OX}{\Omega_X}

\title[W-algebras and CDO at the critical level]{W-algebras and Chiral Differential Operators at the critical level}
\author{ Giorgia Fortuna}

\date{\today}
\begin{document}
\maketitle
\begin{abstract}
 Let $\AL$ be the chiral algebra corresponding to the affine Kac-Moody algebra at the critical level $\gc$. Let
$\ZZ$ be the center of $\AL$. The commutative chiral algebra $\ZZ$
admits a canonical deformation into a non-commutative chiral 
algebra $\WW$. In this paper we will express the resulting first order deformation via the chiral algebra
$\mathcal{D}_{crit}$ of chiral differential operators on $G((t))$ at the critical level.

\end{abstract}

\section{Introduction}\label{intro}
\subsection{}
\noindent Let $X$ be a smooth curve over the complex numbers with structure sheaf $\mathcal{O}_X$ and sheaf of differential operators
 $\DX$.
Let $\g$ be a simple Lie algebra over $\mathbbm{C}$, and let $\kappa$ be an invariant non-degenerate bilinear form $$\kappa:\g\otimes \g\rightarrow \mathbbm{C}.$$
Note that the condition on $\g$ being simple implies that $\kappa$ is a multiple of the Killing form $\kappa_{kill}$.  
Let $\hat{\g}_{\kappa}$ be the affine Kac-Moody algebra given as the central extension of the loop algebra $\g((t))$
\begin{equation}\label{gc}
0\rightarrow \mathbbm{C}\mathbbm{1}\rightarrow \hat{\g}_{\kappa}\rightarrow \g((t))\rightarrow 0,
\end{equation}
with bracket given by
\[
 [af(t),bg(t)]=[a,b]f(t)g(t)+\kappa(a,b)\text{Res}(fdg)\cdot\mathbbm{1},
\]
\noindent where $a$ and $b$ are elements in $\g$, and $\mathbbm{1}$ is the central element.\\
\noindent Our basic tool in this paper is the theory of chiral algebras. In particular we will use
 the chiral algebra $\A_{\kappa}$ attached 
to $\hat{\g}_{\kappa}$ as defined in \cite{AG}. We will assume the reader is familiar with
the foundational work \cite{BD} on this subject. However we will briefly recall some basic definitions and notations.
Throughout this paper $\Delta:X\hookrightarrow X\times X$ will denote the diagonal embedding and
 $j:U\rightarrow X\times X$ its complement, where $U=(X\times X)-\Delta(X)$.\\  For any two sheaves $\M$ and $\mathcal{N}$ 
denote by $\M\boxtimes \mathcal{N}$ the external tensor product 
$\pi_1^*\M\underset{\mathcal{O}_{X\times X}}{\otimes}\pi_2^*\mathcal{N}$, where $\pi_1$ and $\pi_2$ are the 
two projections from $X\times X$ to $X$.  For a right $\DX$-module $\M$ define the extension $\Delta_!(\M)$ as
\[
\Delta_!(\M)=j_*j^*(\Omega_X\boxtimes \M)/\Omega_X\boxtimes \M.
\]
Sections of $\Delta_!(\M)$ can be thought as distributions on $X\times X$ with support on the
 diagonal and with values on $\M$. If $\M$ and $\mathcal{N}$ are two right $\DX$-modules, we will denote by 
$\M\overset{!}{\otimes}\mathcal{N}$ the right $\DX$-module $\M\otimes \mathcal{N}\otimes \Omega_X^*$.\\

\vspace{0.2cm}
\noindent Recall  that a \emph{chiral algebra} over $X$ is a right $\DX$-module $\A$ endowed with a \emph{chiral bracket}, i.e. with a map of 
$\mathcal{D}_{X^2}$-modules 
\[
 \mu:j_*j^*(\A\boxtimes \A)\rightarrow \Delta_!(\A)
\]
which is antisymmetric and satisfies the Jacobi identity.\\
We will denote by $[\,,\,]_{\A}$ the restriction of $\mu$ to $\A\boxtimes \A\hookrightarrow j_*j^*(\A\boxtimes \A)$.\\ 
By a \emph{commutative chiral algebra} we mean a chiral algebra $\R$ such that $[\,,\,]_{\R}$ vanishes. 
In other words it is a chiral algebra such that the chiral bracket $\mu$ factors as
\[
 j_*j^*(\R\boxtimes \R)\rightarrow \Delta_!(\R\overset{!}{\otimes}\R)\rightarrow\Delta_!(\R).
\]
\noindent Equivalently, $\R$ can be described as a right $\DX$-module with a commutative product on the corresponding 
left $\DX$-module $\R^l:=\R\otimes \Omega_X^*$. For instance, in the case $\R=\Omega_X$, with chiral product defined as $\mu(f(x,y)dx\boxtimes dy)=f(x,y)dx\wedge dy \pmod {\Omega^2_{X\times X}}$, you simply recover the commutative product on the sheaf of functions on $X$, which is in fact the left $\DX$-module corresponding to $\Omega_X$.
\vspace{0.3cm}

 \noindent Consider now the chiral algebra $\A_{\kappa}$ attached to $\hat{\g}_{\kappa}$. 
For $\kappa=\kappa_{crit}=-\frac{1}{2}\kappa_{kill}$ denote by $\ZZ$ the center of $\A_{crit}:=\A_{\kappa_{crit}}$.
 This is a commutative chiral algebra with the property that the fiber $\left(\ZZ\right)_x$ over any point $x\in X$  is equal
 to the commutative algebra $\text{End}_{\gc}(\mathbbm{V}_{\g,crit}^{0})$, where 
$$\mathbbm{V}_{\g,crit}^0:=\text{Ind}_{\g[[t]]\oplus \mathbbm{C}}^{\gc}\mathbbm{C}.$$
Note that in the above definition we are using the fact that the exact sequence (\ref{gc}) splits over $\g[[t]]\subset \g((t))$, hence
we can regard the direct sum $\g[[t]]\oplus \mathbbm{C}$ as a subalgebra of $\gc$.\\
The chiral algebra $\ZZ$ is closely related to the center of the twisted enveloping algebra of $\gc$. For any 
chiral algebra $\A $ and any point $x\in X$, we can form an associative topological algebra $\hat{\A}_{x}$
 with the property that its discrete continuous modules 
are the same as $\A$-modules supported at $x$ (see \cite{BD} 3.6.6).
In this case the topological associative algebra corresponding to $\ZZ$ is isomorphic to the center of
 the appropriately completed twisted  enveloping algebra $U^{'}(\gc)$ of $\hat{\g}_{crit}$, where  
$U^{'}(\gc)$ denotes $U(\hat{\g}_{crit})/(\mathbbm{1}- 1)$, (here $1$ denotes the identity element in $U(\gc)$).\\

\vspace{0.2cm}
The importance of choosing the level $\kappa$ to be $\kappa_{crit}$ relies on the fact that the center $\hat{\mathfrak{Z}}_{crit}$ of 
$U^{'}(\gc)$ happens to be very big, unlike any other level $\kappa\neq \kappa_{crit}$ where the center is in fact just $\mathbbm{C}$, as shown in \cite{FF}.
Another crucial feature of the critical level is that  $\hat{\mathfrak{Z}}_{crit}$ carries a natural Poisson structure, obtained by 
considering the one parameter deformation of $\kappa_{crit}$ given by $\kappa_{crit}+\hbar\kappa_{kill}$, as explained below in the 
language of chiral algebras. Moreover, according to \cite{FF, F1, F2}, the center $\hat{\mathfrak{Z}}_{crit}$ is isomorphic, as Poisson 
algebra, to the \emph{quantum Drinfeld-Sokolov reduction} of $U'(\gc)$ introduced in \cite{FF}.
In particular the above reduction provides a quantization of the Poisson algebra $\hat{\mathfrak{Z}}_{crit}$ that will be 
central in this article.\\ Since the language we have chosen is the one of chiral algebras, we will 
now reformulate these properties for the algebra $\ZZ$.\\

The commutative chiral algebra 
$\ZZ$ can be equipped with a Poisson structure which can be described in either of the following two equivalent ways:
\begin{itemize}
\item For any $\hbar\neq 0$ let $\kappa$ be any non critical level $\kappa=\kappa_{crit}+\hbar\kappa_{kill}$ and denote by
$\A_{\hbar}$ the chiral algebra $\A_{\kappa}$. Let $z$ and $w$ be elements of $\ZZ$. 
Let $z_{\kappa}$ and $w_{\kappa}$ be any two families of elements in $\A_{\hbar}$ such that $z=z_{\kappa}$ and $w=w_{\kappa}$ when
 $\hbar=0$. 
Define the Poisson bracket
of $z$ and $w$ to be
\[
\{z,w\}=\frac{[z_{\kappa},w_{\kappa}]_{\A_{ \hbar}}}{\hbar}\pmod {\hbar}.
\]

\item The functor $\Psi$ of semi-infinite cohomology  introduced in \cite{FF} (which is the analogous of the quantum  Drinfeld-Sokolov reduction mentioned before 
and whose main properties will be recalled later),
 produces a 1-parameter family of chiral 
algebras $\{\mathcal{W}_{\hbar}\}:=\{\Psi(\A_{\hbar})\}$  such that $\mathcal{W}_{0}\simeq\ZZ$.
 Define the Poisson structure on $\ZZ$ as $$\{z,w\}=\frac{[\tilde{z}_{\hbar},\tilde{w}_{\hbar}]_{\WW}}{\hbar} \pmod{\hbar}$$
where  $z=\tilde{z}_{\hbar}|_{\hbar=0}$ and  $w=\tilde{w}_{\hbar}|_{\hbar=0}$.
\end{itemize}
Although the above two expressions look the same, we'd like to stress the fact that, unlike the second construction, in the first we are not given any
deformation of $\ZZ$. In other words the elements $z_{\kappa}$ and $w_{\kappa}$ do not belong to the center of $\A_{\kappa}$ (that in fact is trivial).
It is worth noticing that the associative topological algebras $\hat{\mathcal{W}}_{h}$ associated to them
 (usually denoted by $W_{\hbar}$) are the well known \emph{W-algebras}.\\
As in the case of usual algebras, the Poisson structure on $\ZZ$ gives the sheaf of K\"{a}hler differentials
 $\Omega^1(\ZZ)$ a structure of Lie$^*$ algebroid. A remarkable feature, when dealing with chiral algebras, is that the 
existence of a quantization 
$\{\mathcal{W}_{\hbar}\}$ of $\ZZ$ allows us to construct what is called a \emph{chiral extension} $\Omega^c(\ZZ)$ of the 
Lie$^*$ algebroid $\Omega^1(\ZZ)$, and moreover, as it is explained in \cite{BD} 3.9.11,  this establishes an equivalence of 
categories between $1$-st order quantizations of $\ZZ$ and chiral 
extensions of $\Omega^1(\ZZ)$. This equivalence is the point of departure for this paper.\\

\vspace{0.2cm}
We will now 
recall the definitions of the main objects we will be using (a more detailed description can be found in \cite{BD} Section 3.9.).\\
Since we will only deal with right $\DX$-modules, for any two right $\DX-$modules $\M$ and $\mathcal{N}$, we will simply write $\M\otimes \mathcal{N}$ instead of 
$\M\overset{!}{\otimes}\mathcal{N}$. We will also assume that our modules are flat as $\mathcal{O}_X$-modules. Therefore we always have 
an exact sequence 
\[
 0\rightarrow \M\boxtimes \mathcal{N}\rightarrow j_*j^*(\M\boxtimes \mathcal{N})\rightarrow \Delta_!(\M\otimes \mathcal{N})\rightarrow 0,
\]
(note that we are already adopting the different notation for the tensor product).

\vspace{0.3 cm}

\noindent Let $(\R,\,m:\R\otimes \R\rightarrow \R)$ be a commutative chiral algebra.
\begin{defn}\label{lieextension}
Let $L$ be a Lie$^*$ algebra acting by derivations on $\R$ via a map $\tau$. An \emph{$\R$-extension} of $L$ is a 
$\DX$-module $L^c$ fitting in the short exact sequence
\[
0\rightarrow \R\rightarrow L^c\xrightarrow{\pi} L\rightarrow 0
\]
together with a Lie$^*$ algebra structure on $L^c$ such that $\pi$ is a morphism of Lie$^*$ algebras and the adjoint 
action of $L^c$ on $\R\subset L^c$ coincides with $\tau\circ\pi$.
\end{defn}
\begin{defn}
 A Lie$^*$ $\R$-algebroid $\LI$ is a Lie$^*$ algebra with a central action of $\R$ 
( a map $\R\otimes \LI\rightarrow \LI$) and a Lie$^*$ action
 $\tau_{\LI}$ of $\LI$ on $\R$ by derivations such that
 \begin{itemize}
  \item $\tau_{\LI}$ is $\R$-linear with respect to the $\LI$-variable.\\
  \item The adjoint action of $\LI$ is a $\tau_{\LI}$-action of $\LI$ (as a Lie$^*$ algebra) on $\LI$ (as an $\R$-module).
 \end{itemize}

\end{defn}

\noindent In the next definitions we consider objects equipped with a chiral action of $\R$ instead of just a central one. 
\begin{defn}\label{chiral algebroid}
Let $\R$ be a commutative chiral algebra, and $\LI $ be a Lie$^*$ $\R$-algebroid. A \emph{chiral $\R$-extension} of $\LI$ is a 
$\mathcal{D}_{X}$-module $\LI^c$ such that

\begin{equation}\label{AAA}
0\rightarrow \R\xrightarrow{i} \LI^c\rightarrow \LI\rightarrow 0,
\end{equation}
together with a Lie$^*$ bracket and a chiral $\R$-module structure $\mu_{\R,\LI^c}$ on $\LI^c$ satisfying the following properties:
\begin{itemize}
\item 
The arrows in (\ref{AAA}) are compatible with the Lie$^*$ algebra and chiral $\R$-module
structures.\\
\item The chiral operations $\mu_{\R}$ and $\mu_{\R,\LI^c}$ are compatible with the Lie$^*$ actions of $\LI^c$.\\
\item The $*$ operation that corresponds to $\mu_{\R,\LI^c}$ (i.e. the restriction of $\mu_{\R,\LI^c}$ to $\R\boxtimes \LI^c$) is 
equal to $-i\circ \sigma\circ\tau_{\LI^c,\R}\circ \sigma$, where 
$\tau_{\LI^c,\R}$ is the $\LI^c$-action on $\R$ given by the projection $\LI^c\rightarrow \LI$ and the $\LI$ action $\tau_{\LI,\R}$ on $\R$ and $\sigma$ is the transposition of variables. In other words the following diagram commutes
\[
 \xymatrix{\R\boxtimes \LI^c\ar[r]\ar[d]&j_*j^*(\R\boxtimes \LI^c)\ar[r]^-{\mu_{\R,\LI^c}}&\Delta_!(\LI^c).\\
\R\boxtimes \LI\ar[r]^{\sigma\circ \tau_{\LI,\R}\circ\sigma}&\Delta_!(\R)\ar[ur]_{\Delta_!(i)}&}
\]

\end{itemize}
\end{defn}
The previous definition can be extended by replacing $\R$ with any  chiral algebra $\C$ endowed with a central action of $\R$.
More precisely a \emph{ chiral $\C$-extension} of $\LI$ is a $\DX$-module $\LI^c$ such that 
\begin{equation}\label{AAAA}
 0\rightarrow \C\rightarrow \LI^c\rightarrow \LI\rightarrow 0,
\end{equation}
together with a Lie$^*$ bracket and a chiral $\R$-module structure $\mu_{\R,\LI^c}$ on $\LI^c$ such
that:
\begin{itemize}
\item 
The arrows in (\ref{AAAA}) are compatible with the Lie$^*$ algebra and chiral $\R$-module
structures.\\
\item The chiral operations $\mu_{\C}$ and $\mu_{\R,\LI^c}$ are compatible with the Lie$^*$ actions of $\LI^c$.\\
\item The structure morphism $\R\rightarrow \C$ is compatible with the Lie$^*$ actions of $\LI^c$.\\
\item The $*$ operation that corresponds to $\mu_{\R,\LI^c}$ (i.e. $\mu_{\R,\LI^c}$ restricted to $\R\boxtimes \LI^c$) is equal to $-i\circ \sigma\circ\tau_{\LI^c,\R}\circ \sigma$, where $\tau_{\LI^c,\R}$ is the $\LI^c$-action on $\R$, $\sigma$ is the transposition of variables and $i$
is the composition of the structure morphism $\R\rightarrow \C$ and the embedding $\C\subset \LI^c$.
\end{itemize}

\begin{defn}\label{chiral envelope}
The \emph{chiral envelope} of the chiral extension $(\R,\C,\LI^c,\LI)$ is a pair $(U(\C,\LI^c),\phi^c)$, 
where $U(\C,\LI^c)$ is a chiral algbera  and $\phi^c$ is a homomorphism of $\LI^c$ into $U(\C,\LI^c)$, 
satisfying the following universal property. For every chiral algebra $\mathcal{A}$ and any morphism $f:\LI^c\rightarrow \mathcal{A}$ such that:
\begin{itemize}\itemsep -2pt
\item $f$ is a morphism of Lie$^*$ algebras.\\
 \item $f$ restricts to a morphism of chiral 
algebras on $\C\subset \LI^c$.\\
\item $f$ is a morphism of $\R$-modules (where the $\R$-action on $\mathcal{A}$ is the one given by the above point),
\end{itemize}
 there exist a unique map $\overline{f}:U(\C,\LI^c)\rightarrow \mathcal{A}$ that makes the following diagram commutative
\[
 \xymatrix{\LI^c\ar[r]^{f}\ar[d]^-{\phi^c}&\mathcal{A}\\
U(\C,\LI^c)\ar[ur]_{\overline{f}}&}.
\]
It is shown in \cite{BD} that such object exists. When $\C=\R$ we will simply write $U(\LI^c)$ instead of $U(\R,\LI^c)$.

\end{defn}

\subsection{}\label{VVV}{\scshape Chiral extensions of $\Omega^1(\R)$.} Let $\R$ be a commutative chiral algebra equipped with 
a Poisson bracket $\{\,,\,\}:\R\boxtimes \R\rightarrow \Delta_!(\R)$. We will refer to such object as \emph{chiral-Poisson algebra}.
  As we mentioned before, any Poisson structure on a commutative chiral algebra $\R$ gives the sheaf $\OO$ a structure of a 
Lie$^*$ algebroid. 
Now consider the following: to a chiral extension 
\[
0\rightarrow \R\rightarrow \Omega^c(\R)\rightarrow \OO\rightarrow 0,
\]
consider the pull-back of the above sequence via the differential $d:\R\rightarrow \OO$. 
 The resulting short exact sequence is a $\mathbbm{C}[\hbar]/\hbar^2$-deformation of the chiral-Poisson algebra $\R$\footnote{If $\{\,,\,\}$ denotes the Poisson bracket on $\R$, this is indeed a quantization of ($\R, 2\{\,,\,\}$)}.
If we denote by $\mathcal{Q}^{ch}(\R)$ the groupoid of $\mathbbm{C}[\hbar]/\hbar^2$-deformations of the chiral-Poisson algebra $\R$, and by $\mathcal{P}^{ch}(\OO)$ the groupoid of chiral $\R$-extensions of $\OO$, 
the above map defines a functor 
\[
 \mathcal{P}^{ch}(\OO)\rightarrow \mathcal{Q}^{ch}(\R).
\]
 In \cite{BD} 3.9.10. the following is shown.
\begin{thm}\label{bbdd}
 The above functor defines an equivalence between $\mathcal{P}^{ch}(\OO)$ and $\mathcal{Q}^{ch}(\R)$.
\end{thm}
\noindent Consider now the following diagram:
\[
\xymatrix{\mathcal{P}^{ch}(\OO)\ar[r]\ar[d]^{\simeq}& \left\{ \text{Lie$^*$ algebroid structures on $\OO$}\right\}\ar[d]^{\simeq}\\
\mathcal{Q}^{ch}(\R)\ar[r]&\left\{ \text{Chiral-Poisson structures on $\R$}\right\}.
} 
\]
A natural question to ask is the following: if we consider the quantization of $\ZZ$ given by
 applying the functor of semi-infinite cohomology to $\A_{\hbar}$,
how does the corresponding chiral extension of $\Omega^1(\ZZ)$ look like? 
The answer to the above question is the main body of this paper.\\
We will give
 an explicit construction of $\Omega^c(\R)$ for an arbitrary chiral-Poisson algebra $\R$. 
In the case where $\R=\ZZ$ we will see how this chiral extension relates to the \emph{chiral algebra of differential 
operators on the loop group $G((t))$} at the critical level introduced in \cite{AG}, where $G$ 
is the algebraic group of adjoint type corresponding to $\g$.\\
\subsection{}{\scshape Main Theorem.} \noindent Recall from \cite{AG} the definition of the chiral algebra of twisted differential operators on $G((t))$. Denote such algebra by $\mathcal{D}_{\hbar}$
(note that this notation differs from the one in \cite{AG} where the same object was denoted by $\mathcal{D}_{\kappa}$, with
$\kappa=\kappa_{crit}+\hbar\kappa_{kill}$). 
As it is explained there, the fiber  $\left(\mathcal{D}_{\hbar}\right)_x$ of $\DDH$ at $x\in X$ is isomorphic to
\[
 \left(\mathcal{D}_{\hbar}\right)_x\simeq U(\hat{\mathfrak{g}}_{\kappa})\underset{U(\g[[t]]\oplus \mathbbm{C})}{\otimes}\mathcal{O}_{G[[t]]}.
\]
Moreover $\mathcal{D}_{\hbar}$ comes equipped with two embeddings 
\begin{equation}\label{EM}
 \LAH\xrightarrow{l_{\hbar}}\DDH\xleftarrow{r_{\hbar}}\LAHH
\end{equation}
corresponding to left and right invariant vector fields on the loop group $G((t))$.
 In particular, for $\hbar=0$, we have $\LAH=\LAHH=\AL$ and $\DDH=\DDC$. Therefore we
obtain two different embeddings, $l:=l_0$ and $r:=r_0$ of $\AL$ into $\DDC$
\[
\AL\xrightarrow{l}\DDC\xleftarrow{r}\AL. 
\]
\noindent If we restrict these two embeddings to $\ZZ$, as it is explained in \cite{FG} Theorem 5.4, we have
\[
 l(\ZZ)=l(\AL)\cap r(\AL)=r(\ZZ).
\]
Moreover the two compositions
\[
 \ZZ\hookrightarrow \AL\xrightarrow{l}\DDC\xleftarrow{r}\AL\hookleftarrow \ZZ
\]
are intertwined by the automorphism $\eta: \ZZ\rightarrow \ZZ$ given by
 the involution of the Dynkin diagram that sends a weight $\lambda$ to $-w_0(\lambda)$ (i.e. when restricted to $\ZZ$ we 
have $l=r\circ \eta$).\\
\subsection{}The two embedding $l$ and $r$ of $\A_{crit}$ into $\DDC$ endow the fiber $(\DDC)_x$ with a structure of $\gc$-bimodule. The fiber can therefore be decomposed according to these actions as explained below.\\

Denote by $\hat{\mathfrak{Z}}_{crit}$ the center of the completion of the 
enveloping algebra of $\gc$. For a dominant weight $\lambda$, let $V^{\lambda}$ be the finite dimensional irreducible representation of 
$\g$ with highest weight $\lambda$ and let $\mathbbm{V}^{\lambda}_{\g,crit}$ be the $\gc$-module given by 
\[
 \mathbbm{V}^{\lambda}_{\g,crit}:=U(\gc)\underset{U(\g[[t]]\oplus\mathbbm{C})}{\otimes}V^{\lambda}.
\]
\noindent The action of the center $\hat{\mathfrak{Z}}_{crit}$ on $\mathbbm{V}_{\g,crit}^{\lambda}$ factors as follows
\[
\hat{\mathfrak{Z}}_{crit} \twoheadrightarrow \mathfrak{z}_{crit}^{\lambda}:=\End(\mathbbm{V}_{\g,crit}^{\lambda}). 
\]

\noindent Denote by $I^{\lambda}$ the kernel of the above map, and consider the
formal neighborhood of $\text{Spec}(\mathfrak{z}_{crit}^{\lambda})$ inside $\text{Spec}(\hat{\mathfrak{Z}}_{crit})$.
Let $\hat{\g}_{crit}$-mod$^{G[[t]]}$ be the full subcategory of $\gc$-modules such that the action of $\g[[t]]$ can be integrated to an action of $G[[t]]$.
We have the following Lemma.
\begin{lem}
Any module $M$ in $\gc$-mod$^{G[[t]]}$ can be decomposed into a direct sum of submodules
 $M_{\lambda}$ such that each $M_{\lambda}$ admits a filtration whose subquotients are annihilated by $I^{\lambda}$. 
\end{lem}
As a bimodule over $\gc$ the fiber at any point $x\in X$ of $\DDC$ is $G[[t]]$ integrable with respect to both actions, hence we have 
two direct sum decompositions of $\left(\mathcal{D}_{crit}\right)_x$ corresponding to the left and right action of $\gc$. These decompositions
coincide up to the involution $\eta$ and we have
\[
\left( \mathcal{D}_{crit}\right)_x=\underset{\lambda \,dominant}{\bigoplus}\left(\mathcal{D}_{crit}\right)_x^{\lambda},
\]
where $\left(\mathcal{D}_{crit}\right)_x^{\lambda}$ is the direct summand supported on the formal completion of 
$ \text{Spec}(\mathfrak{z}_{crit}^{\lambda})$.\\
Denote by $\DD$ the $\mathcal{D}_{X}$-module corresponding to $\left(\mathcal{D}_{crit}\right)_x^{0}$. It is easy to see that $\DD$ is in fact a 
chiral algebra.\\
Since the fiber of $\AL$ at $x$ is isomorphic to $\mathbbm{V}^{0}_{\g,crit}$, the embeddings 
$l$ and $r$ must land in the chiral algebra $\DD$. Hence we have 
\[
 \AL\xrightarrow{l,\,r}\DD\longhookrightarrow \DDC.
\]
The above two embeddings give $\DD$ a structure of $\AL$-bimodule, hence it makes sense to apply
 the functor of semiinfinite cohomology $\Psi$ to it twice (as it will explained in \ref{filtration1}).
 Let us denote by $\B$ the resulting chiral algebra
\[
 \B:=(\Psi\boxtimes \Psi)(\mathcal{D}_{crit}^0).
\]
The main result of this paper is the following.
\newtheorem{Theorema}{Theorem}
\begin{Theorema}\label{main}
 The chiral envelope $U(\Omega^c(\ZZ))$ of the extension 
$$0\rightarrow \ZZ\rightarrow\Omega^c(\ZZ)\rightarrow \Omega(\ZZ)\rightarrow 0,$$
given by the quantization $\{\mathcal{W}_{\hbar}:=\Psi(\A_{\hbar})\}$ of the center $\ZZ$,
is isomorphic to the chiral algebra $\B$.
\end{Theorema}
\subsection{}{\scshape Structure of the proof.} The proof of Theorem \ref{main} will be organized as follows: in Section \ref{reformulation}
we will give an alternative formulation of the Theorem that consists in finding a map $F$ from $\Omega^c(\ZZ)$ to $\B$ with some
particular properties. The definition
of the above map will rely on the explicit construction of the chiral extension $\Omega^c(\ZZ)$ that will be given in Section \ref{conomega}.
In Section \ref{mapF} we will finally define the map $F$ and conclude the proof of the Theorem.
\subsection{}{\scshape   Acknowledgements.}
The author would like to thank her advisor Dennis Gaitsgory for suggesting the problem, for  his guidance through this work and for the infinitely many helpful conversations.  The author is also grateful to M. Artin and D. Gaitsgory for their help in improving the exposition of this work. 
The author would also like to thank Andrea Appel, Salvatore Stella, Sam Raskin and Sasha Tsymbaliuk for their help and support.
\section{Reformulation of the Theorem}\label{reformulation}
 
\subsection{}
In this section we will show how to prove the Theorem \ref{main} assuming the existence of a map $F$ from $\Omega^c(\ZZ)$ to $\B$. In order to do
so, we will use the fact that both $U(\Omega^c(\ZZ))$ and $\B$ can be equipped with filtrations as explained below.
\subsection{}
\noindent  The chiral algebra $U(\Omega^c(\ZZ))$, being the chiral envelope of the extension
\[
 0\rightarrow \ZZ\rightarrow \Omega^c(\ZZ)\rightarrow \Omega^1(\ZZ)\rightarrow 0,
\]
has its standard \emph{Poincar\'{e}-Birkhoff-Witt filtration}. In fact, more generally, given a chiral-extension $(\R,\C,\LI^c,\LI)$, using the notations from Definition \ref{chiral envelope},
we can define a PBW filtration on $U(\C,\LI^c)$ as the filtration  generated by $U(\C,\LI^c)_{0}:=\phi^c(\C)$ and
 
\begin{eqnarray*}
U(\C,\LI^c)_1:=\text{Im}(j_*j^*(\LI^c\boxtimes \C)\xrightarrow{\phi^c\boxtimes \phi^c|_{\C}}\\
\rightarrow j_*j^*(U(\C,\LI^c)\boxtimes U(\C,\LI^c))\rightarrow \Delta_!(U(\C,\LI^c))).
\end{eqnarray*}

\vspace{0.2cm}
Moreover in \cite{BD} 3.9.11. the following theorem is proved.
\begin{thm}\label{GR}
 If $\R$ and $\C$ are $\mathcal{O}_{X}$ flat and $\LI$ is a flat $\R$-module then we have an isomorphism
\[
 \C\underset{\R}{\otimes}\text{Sym}_{\R}^{\cdot}\LI\overset{\simeq}{\rightarrow}\text{gr}_{\cdot}U(\C,\LI^c).
\]
\end{thm}
By applying the above to the case where $\C=\R=\ZZ$ and the extension of $\LI=\Omega^1(\ZZ)$ given by $\LI^c=\Omega^c(\ZZ)$ we get
\[
 \text{gr}_{\cdot}U(\Omega^c(\ZZ))\simeq \text{Sym}^{\cdot}_{\ZZ}\Omega^1(\ZZ).
\]
\subsection{}\label{filtration1} The filtration on $\B$ is defined using the functor $\Psi$ of 
\emph{semi-infinite cohomology} introduced in \cite{FF}.\\

Recall that, for any central charge $\kappa=\hbar\kappa_{kill}+\kappa_{crit}$,
the functor $\Psi$ assigns to a chiral $\A_{\hbar}$-module a $\Psi(\A_{\hbar})=\WW$-module. 
In particular, for every chiral algebra $\mathcal{B}$, and every morphism of chiral algebras $\phi: \A_{\hbar}\rightarrow \mathcal{B}$ we have 
\[
\Psi:\left\{\begin{array}{c}\text{chiral algebra morphism }\\ \phi:\A_{\hbar}\rightarrow \mathcal{B}\end{array} \right\}\rightarrow \left\{\begin{array}{c}\text{chiral algebra morphism }\\ \Psi(\phi): \WW\rightarrow \Psi(\mathcal{B})\end{array} \right\}. 
\]
\noindent Moreover recall that for $\hbar=0$ we have $\Psi(\A_{crit})\simeq \ZZ.$

\vspace{0.3 cm}
\noindent As it is explained in \cite{FG}, the chiral algebra $\B$ can be described as  
\[
(\Psi\boxtimes \Psi)(U(\C,\LI^c))\xrightarrow{\sim}\B=(\Psi\boxtimes \Psi)(\mathcal{D}^{0}_{crit}),
\]
for some particular chiral algebra $\C$ and chiral extension $\LI^c$. Hence it carries a canonical filtration induced by the PBW-filtration on $U(\C,\LI^c)$. 
We will recall below the definitions of these algebras.\\

\vspace{0.1cm}
\noindent {\scshape The renormalized chiral algebroid.} Recall that \cite{FG} Proposition 4.5. shows the existence of a chiral extension $\A^{ren,\tau}$
 that fits into the following exact sequence
\[
 0\rightarrow (\A_{crit}\underset{\ZZ}{\otimes} \A_{crit})\rightarrow \A^{ren,\tau}\rightarrow \Omega^1(\ZZ)\rightarrow 0,
\]
which is a chiral extension of $(\A_{crit}\underset{\ZZ}{\otimes} \A_{crit}) $ in the sense we introduced in Definition 
\ref{chiral algebroid}. In particular, if we consider the chiral envelope $U((\A_{crit}\underset{\ZZ}{\otimes} \A_{crit}), \A^{ren,\tau})$, by Theorem \ref{GR} we have
\begin{eqnarray*}
 &\text{gr}_{\cdot}(U((\A_{crit}\underset{\ZZ}{\otimes} \A_{crit}),\,\A^{ren,\tau}))\simeq \\
&\simeq( \AL\underset{\ZZ}{\otimes} \AL )\underset{\ZZ}{\otimes}\text{Sym}^{\cdot}_{\ZZ}(\Omega^1(\ZZ)).
\end{eqnarray*}
The chiral envelope $U((\A_{crit}\underset{\ZZ}{\otimes} \A_{crit}), \A^{ren,\tau})$ is closely related to the chiral algebra $\mathcal{D}^0_{crit}$, in fact in \cite{FG} the following is proved:
\begin{thm}\label{REN}
We have an embedding $G$ of the chiral extension $\A^{ren,\tau}$ into $\DDC$ such that the maps $l$ and $r$ are the compositions of this embedding 
with the canonical maps
\[
 \AL\rightrightarrows (\A_{crit}\underset{\ZZ}{\otimes} \A_{crit})\xrightarrow{G} \A^{ren,\tau}.
\]
The embedding extends to a homomorphism of chiral algebras $$U((\A_{crit}\underset{\ZZ}{\otimes} \A_{crit}),\, \A^{ren,\tau})\rightarrow \DDC$$ and the latter is an isomorphism into $\DD$.
\end{thm}
\noindent Therefore we see that $\B$ is given by applying the functor $\Psi\boxtimes \Psi$ to the chiral envelope $U(\C,\LI^c),$ for 
$$\C=(\A_{crit}\underset{\ZZ}{\otimes} \A_{crit}),\,\,\, \text{and}\,\,\,\LI^c=\A^{ren,\tau}.$$ 
In particular, since the functor $\Psi$
is exact, we obtain a filtration on $\B$ induced from the PBW-filtration on $U(\AL\otimes \AL, \A^{ren,\tau})$ such that
\[
 \text{Sym}^{\cdot}_{\ZZ}\Omega^1(\ZZ)\xrightarrow{\sim}\text{gr}_{\cdot}\B,
\]
where we used the fact that $\Psi(\A_{crit})\simeq \ZZ$.
\subsection{}\label{embedding}
 Note that if we apply the functor $\Psi$ to the two embeddings in (\ref{EM}), we obtain two embeddings 
\[
\mathcal{W}_{\hbar}\xrightarrow{l_{\hbar}}(\Psi\boxtimes \Psi)(\mathcal{D}_{\hbar})\xleftarrow{r_{\hbar}}\mathcal{W}_{-\hbar}
\]
such that $l:=l_0=r_0\circ \eta=:r\circ \eta$, 
where we are denoting simply by $l_{\hbar}$ and $r_{\hbar}$ the maps $\Psi(l_{\hbar})$ and $\Psi(r_{\hbar})$ respectively.
In particular, for $\hbar=0$, we obtain two 
embeddings $l$ and $r$
of $\ZZ$ into $(\Psi\boxtimes \Psi)(\mathcal{D}_{crit})$ that differs by $\eta$. Moreover the image of the two 
maps lands in $\B$, therefore we obtain two embeddings
\[
\ZZ\xrightarrow{l}\B\xleftarrow{r}\ZZ. 
\]
From the above construction it is clear that $\ZZ$ corresponds to the 0-th part of the filtration
defined on $\B$. Moreover, by the definition of the map $G$ from Theorem \ref{REN} (see \cite{FG}), the embedding
 $\ZZ\hookrightarrow \B$ induced by the inclusion $(\A_{crit}\underset{\ZZ}{\otimes} \A_{crit})\hookrightarrow U((\A_{crit}\underset{\ZZ}{\otimes} \A_{crit}), \A^{ren,\tau})$ under $\Psi\boxtimes \Psi$,  coincides with $l$.

\subsection{} Suppose now that we are given a map  $F:\Omega^c(\ZZ)\rightarrow \B$ satisfying the conditions 
stated in Definition \ref{chiral envelope}.
 By the universal property of the chiral envelope, we automatically get
a map
\[
 U(\Omega^c(\ZZ))\rightarrow \B.
\]
Clearly not every such map will induce an isomorphism between the two chiral algebras.
Theorem \ref{main} can be reformulated as saying that there exists a map as above, 
that gives rise to an isomorphism $U(\Omega^c(\ZZ))\xrightarrow{\sim} \B$.
More precisely we have the following: 

\begin{Theorema}\label{main2}
 There exists a map $F:\Omega^c(\ZZ)\rightarrow (\B)_1\hookrightarrow \B$ compatible with the $\ZZ$ structure on both sides that restricts to the embedding $l$ of chiral algebras
 on $\ZZ$
such that the following diagram commutes 
\[
 \xymatrix{\Omega^c(\ZZ)/\ZZ\ar[rr]^{F}&&(\B)_1/\ZZ.\\
&\Omega^1(\ZZ)\ar[ul]^{\simeq}\ar[ur]_{\simeq}&}
\]
\end{Theorema}
\noindent We will now show how Theorem \ref{main} follows from Theorem \ref{main2}. The proof Theorem \ref{main2} will occupy the rest
of the article.
\begin{proof}[Proof of (Theorem \ref{main2} $\Rightarrow$ Theorem \ref{main}).]
 To prove Theorem \ref{main} we need to  show that the above $F$ induces an isomorphism $U(\Omega^c(\ZZ))\xrightarrow{\sim} \B$. 
This amounts to showing that the following diagram commutes for every $i$:
\begin{equation}\label{dia}
 \xymatrix{\text{gr}_{i+1} U(\Omega^c(\ZZ))\ar[rr]^{F}&&\text{gr}_{i+1} \B.\\
&\text{Sym}_{\ZZ}^{i+1} \Omega^1(\ZZ)\ar[ul]^{\simeq}\ar[ur]_{\simeq}&}
\end{equation}
But this follows from the fact that the above filtrations are generated by their first two terms. 
In fact, more generally, for any chiral envelope $U(\LI^c)$, we have
\[
\Delta_!(\text{gr}_{i+1} U(\LI^c)):=
\]

\[
=\bigslant{\text{Im}\left(\begin{array}{c}j_*j^*(U(\LI^c)_1\boxtimes U(\LI^c)_{i})\rightarrow\\ \Delta_!(U(\LI^c))\end{array}\right)}{\text{Im}\left(\begin{array}{c}j_*j^*(U(\LI^c)_1
\boxtimes U(\LI^c)_{i-1})\rightarrow \\ \Delta_!(U(\LI^c))\end{array}\right)}.
\]
It is not hard to see that the isomorphism $\text{Sym}^{i+1}_{\ZZ}\Omega^1(\ZZ)\xrightarrow{\sim}\text{gr}_{i+1} U(\Omega^c(\ZZ))$
(and similarly for $\B$) is the one induced by the map 
\[
 j_*j^*(\Omega^1(\ZZ)\boxtimes \text{Sym}^{i}_{\ZZ})\rightarrow \Delta_!(\text{gr}_{i+1} U(\Omega^c(\ZZ))),
\]
that in fact vanishes when restricted to $\Omega^1(\ZZ)\boxtimes \text{Sym}^{i}_{\ZZ}$, and factors through the action of $\ZZ$. 
Therefore the diagram (\ref{dia}) commutes by induction on $i$. 

\end{proof}

\section{construction of the chiral extension $\Omega^c(\ZZ)$}\label{conomega}
\subsection{}
As we saw in the previous section, the proof of Theorem \ref{main2} amounts to the construction 
of a particular map $F$ from $\Omega^c(\ZZ)$ to $\B$.
However the construction of such map
requires a more explicit description of $\Omega^c(\ZZ)$.\\
In fact recall that, for every commutative-Poisson chiral algebra $\R$ and quantization
$\{\R_h\}$ of the Poisson structure, there is canonically associated a chiral extension
\[
 0\rightarrow \R\rightarrow \OC\rightarrow \OO\rightarrow 0
\]
given by the equivalence of category stated in Theorem \ref{bbdd}. However the proof of this theorem doesn't provide a construction of it. This section will be devoted to the construction of the above extension.\\
\subsection{}
Starting from the Lie$^*$ algebra extension 
\[
 0\rightarrow \R\rightarrow \R^c\rightarrow \R\rightarrow 0,
\]
where $\RC:=\R_{\hbar}/\hbar^2\R_{\hbar} $ acts on $\R$ via the projection $\R^c\rightarrow \R$ and the Poisson bracket on $\R$, we will first construct a chiral extension (see Definition \ref{chiral algebroid}) $\INDR$ fitting into 
\[
 0\rightarrow \R\rightarrow \INDR \rightarrow \R\otimes \R\rightarrow 0,
\]
where $\R\otimes \R$ is viewed as a Lie$^*$ algebroid using the Poisson structure on $\R$.
The chiral extension $\OC$ will be then defined as a quotient $\INDR$.\\
More generally, in \ref{inizio}-\ref{fine} we will explain how to  construct a chiral extension $\INDLD$ fitting into
\begin{equation}\label{ext}
 0\rightarrow \R\rightarrow \INDLD\rightarrow \R\otimes L\rightarrow 0
\end{equation}
for every Lie$^*$ algebra $L$ acting on $\R$ by derivations and every extension
\[
 0\rightarrow \R\rightarrow L^c\rightarrow L\rightarrow 0.
\]
\noindent The case where $L=\R$ and $L^c=\R^c$, will be presented in \ref{particular} as a particular case of the above general construction.

\subsection{}\label{inizio} {\scshape Definition of $\INDLD$.}
Let $(\R,\mu)$ be a commutative chiral algebra and let $L$ be a Lie$^*$ algebra acting on $\R$ by derivations via the map
 $\tau$. The induced $\R$-module $\R\otimes L$ has a unique structure of Lie$^*$ $\R$-algebroid 
such that the morphism $1_{\R}\otimes id_{L}:L\rightarrow \R\otimes L$ is a morphism of Lie$^*$ algebras compatible with 
their actions on $\R$. Note that we have an obvious map $$i:L\rightarrow \R\otimes L.$$ The Lie$^*$ algebroid $\R\otimes L$ is called \emph{rigidified}. 
More generally we have the following definition.
\begin{defn}\label{liestar}
 A Lie$^*$ algebroid $\LI$ is called \emph{rigidified} if we are given a Lie$^*$ algebra $L$ acting on $\R$ via the map $\tau$, and an inclusion $i:L\rightarrow \LI$,
such that $\R\otimes L\overset{\sim}{\rightarrow}\LI.$
\end{defn}
\subsection{}
Let $\LI$ be a rigidified Lie$^*$ algebroid. 
Consider the map that sends a chiral extension
 of $\LI$ 
\[
0\rightarrow \R\rightarrow \LI^c\rightarrow \LI\rightarrow 0
\]
to the $\R$ extension of $L$ given by considering the pull-back of the map $i:L\rightarrow \LI$.\\
\noindent Denote by
$\mathcal{P}^{cl}(\LI)$ (resp. $\mathcal{P}^{ch}(\LI)$) the groupoid of classical (resp. chiral) extensions of $\LI$ (where by classical we mean extensions in the category of
Lie$^*$ algebroids), and by $\mathcal{P}(L,\tau)$
the Picard groupoid of $\R$-extensions of $L$.
\noindent Clearly the map mentioned above (that can be equally defined for classical extensions as well), defines two functors 
\[
 \mathcal{P}^{cl}(\LI)\rightarrow \mathcal{P}(L,\tau),\,\,\,\, \mathcal{P}^{ch}(\LI)\rightarrow \mathcal{P}(L,\tau).
\]
As it is explained in \cite{BD} 3.9.9. the following is true.
\begin{prop}\label{hh}
If $L$ is $\mathcal{O}_X$ flat, then these maps define an equivalence of groupoids 
\begin{equation}\label{inverse functor}
 \mathcal{P}^{cl}(\LI)\overset{\sim}{\rightarrow} \mathcal{P}(L,\tau),\,\,\,\, \mathcal{P}^{ch}(\LI)\overset{\sim}{\rightarrow} \mathcal{P}(L,\tau),
\end{equation}
\end{prop}
\vspace{0.2cm}

  Given a Lie$^*$ algebra extension $0\rightarrow \R\rightarrow L^c\rightarrow L\rightarrow 0,$
define $\INDC$ (resp. $\INDLD$) to be the classical (resp. chiral) extension corresponding to the above  sequence under the equivalences stated in the above proposition.\\

\vspace{0.1cm}
\subsection{}In subsections \ref{classic}-\ref{classic2} we will briefly recall the construction of the inverse functors to (\ref{inverse functor}) in the classical and chiral setting respectively (as presented in \cite{BD}). However in \ref{fine} we will give a different construction of the inverse functor in the chiral setting, i.e. a different construction of the chiral extension $\INDLD$ associated to any $\R$-extension of $L$. The latter construction will be used to define the chiral extension $\Omega^c(\R)$. 

\subsection{}\label{classic}
For the "classical" map $ \mathcal{P}^{cl}(\LI)\rightarrow \mathcal{P}(L,\tau)$, to an extension 
\begin{equation}\label{BO}
0\rightarrow \R\rightarrow L^c\rightarrow L\rightarrow 0,
\end{equation}
the inverse functor associates the classical extension $\INDC$ of the Lie$^*$ algebroid $\R\otimes L=\LI$ given by the push-out of the extension
\[
0\rightarrow \R\otimes \R\rightarrow \R\otimes L^c\rightarrow \R\otimes L\rightarrow 0
\]
via the map $m:\R\otimes \R\rightarrow \R$.\\

\vspace{0.2cm}
\subsection{}\label{classic2}
The construction of the inverse functor in the "chiral" setting given in \cite{BD}  (i.e. the construction of $\INDLD$), uses the following two facts:
\begin{itemize}\itemsep -2pt
\item  $\mathcal{P}^{ch}(\LI)$ has a structure 
of  $\mathcal{P}^{cl}(\LI)$-torsor under Baer sum. \\
\item $\mathcal{P}^{ch}(\LI)$ is non empty.
\end{itemize}
The first fact follows from condition $3)$ in the definition of chiral $\R$-extension, which guarantees that the Baer difference 
of two chiral extensions is a classical one. In other words the action of $\R$ on the sum of two chiral extensions is automatically central.\\

\vspace{0.1cm}
\noindent The non emptiness of $\mathcal{P}^{ch}(\LI)$ follows from the existence of a distinguished 
chiral $\R$-extension $\INDL$ attached to every Lie$^*$ algebra $L$ acting on $\R$. Such object is defined by the following:
\begin{defprop}\label{defp}
 Suppose that we are given a Lie$^*$ algebra $L$ acting by derivations on $\R$ via the map $\tau$, and let $\LI$ be a rigidified 
Lie$^*$ algebroid (see Definition \ref{liestar}), so we have a morphism of Lie$^*$ algebras $i:L\rightarrow \LI$ such that $\R\otimes L\xrightarrow{\sim} \LI$.
Then there exist a chiral extension $\INDL$ equipped with a lifting $\overline{i}:L\rightarrow \INDL$ such that $\overline{i}$ is a morphism of Lie$^*$ algebras
and the adjoint action of $L$ on $\R$ via $\overline{i}$ equals $\tau$. The pair $(\INDL,\overline{i})$ is unique.
\end{defprop}

The proof of this proposition can be found in \cite{BD} 3.9.8. However in \ref{pppp} we will recall the construction of $\INDL$ and of the map 
$\overline{i}:L\rightarrow \INDL$. \\

To finish the construction of $\INDLD$ (or in other words, the construction of the inverse to the functor $\mathcal{P}^{ch}(\R)\rightarrow \mathcal{P}(L,\tau)$), we use the classical extension $\INDC$ given in \ref{classic} together with the $\mathcal{P}^{cl}(\LI)$-action on $\mathcal{P}^{ch}(\LI)$. To the extension $0\rightarrow \R\rightarrow L^c\rightarrow L\rightarrow 0$ 
we associate the chiral $\R$-extension $$\INDLD:=\INDC\underset{Baer}{+}\INDL$$ 
\noindent of $\R\otimes L$ by $\R$, where $\INDL$ is the distinguished classical extension defined in \ref{defp}. Note that, after pulling back the extension 
\[
0\rightarrow \R\rightarrow \INDLD\rightarrow \LI\simeq \R\otimes L\rightarrow 0
\] via the map $L\rightarrow \LI\simeq \R\otimes L$, we obtain the Baer sum of the trivial extension (corresponding to $\INDL$) with $L^c$, i.e. we recover the initial Lie$^*$ extension $0\rightarrow \R\rightarrow L^c\rightarrow L\rightarrow 0$ as we should.\\

\subsection{}\label{pppp}
In this subsection we want to recall the construction and the main properties of the distinguished chiral extension $\INDL$ given by Definition-Proposition \ref{defp}.\\

\vspace{0.1cm}
Given a Lie$^*$ algebra $L$ acting on $\R$ by derivations, we can consider the action map $\R\boxtimes L\rightarrow \Delta_{!}(\R)$  and consider the following push out:
\[
\xymatrix{
0\ar[r]&\R\boxtimes L\ar[r]\ar[d]&j_*j^*(\R\boxtimes L)\ar[r]\ar[d]&\Delta_!(\R\otimes L)\ar[r]\ar[d]&0\\
0\ar[r]& \Delta_!(\R)\ar[r]&\Delta_!(\R)\oplus j_*j^*(\R\boxtimes L)/\R\boxtimes L\ar[r]& \Delta_!(\R\otimes L)\ar[r]&0.
} 
\]
The term in the middle is a $\DX$-module supported on the diagonal, hence by Kashiwara's Theorem (see \cite{K} Theorem 4.30) it corresponds to a $\DX$ module on $X$.
This $\DX$-module has a structure of chiral extension and will be our desired $\INDL$ (i.e. we have $\Delta_!(\INDL)\simeq \Delta_!(\R)\oplus j_*j^*(\R\boxtimes L)/\R\boxtimes L$).
\begin{remark}
By construction we have inclusions $\R\rightarrow \INDL$ and a lifting  $\overline{i}:L\rightarrow \INDL$ of $i:L\rightarrow \R\otimes L$.
 In fact we can consider the 
following diagram 
\[
  \xymatrix{0\ar[r]&\OX\boxtimes L\ar[r]\ar[d]& j_*j^*(\OX\boxtimes L)\ar[r]\ar[d]&\Delta_!(\OX\overset{!}{\otimes} L)\simeq \Delta_!(L)\ar[ddl]_{\Delta_!(\overline{i})}\ar[r]\ar[d]^{\Delta_!(i)=\Delta_!(unit\otimes id)}&0\\
0\ar[r]&\R\boxtimes L\ar[r]\ar[d]& j_*j^*(\R\boxtimes L)\ar[r]|\hole\ar[d]^{\pi}&\Delta_!(\R\overset{!}{\otimes}L)\ar[r]\ar[d]&0\\
0\ar[r]&\Delta_!(\R)\ar[r]&\bigslant{\Delta_{!}(\R)\oplus j_*j^*(\R\boxtimes L)}{\R\boxtimes L}\ar[r]&\Delta_!(\R\overset{!}{\otimes} L)\ar[r]& 0.
}
\]
By looking at the composition of the two vertical arrows in the middle, it is not hard to see that this composition factors
through $\Delta_!(L)$. In fact the most left vertical arrow from $\OX\boxtimes L$ to $\Delta_!(\R)$ is zero.
We define $\overline{i}$ to be the map corresponding (under the Kashiwara's equivalence) to $\Delta_!(\overline{i})$.\\

\vspace{0.3 cm}
\noindent As it is shown in \cite{BD} 3.3.6. the inclusions $\R\rightarrow \INDL$, $\overline{i}:L\rightarrow \INDL$ and the chiral operation 
$j_*j^*(\R\boxtimes L)\rightarrow \Delta_!(\INDL)$, 
uniquely determine a chiral action of $\R$ on $\INDL$ and a Lie$^*$ bracket on it. In other words they give $\INDL$ a structure of chiral $\R$-extension.

\end{remark}
\noindent Note that this chiral $\R$-extension corresponds, under the equivalence given by Theorem \ref{hh} (i.e. after we pull-back the extension via the map $\psi:L\rightarrow \R\otimes \LI$),
 to the trivial extension of $L$ by $\R$ in $\mathcal{P}(L,\tau)$.
 \subsection{}\label{summary}
To summarize we have seen that:
\begin{itemize}
\item If a Lie$^*$ algebra $L$ acts on $\R$ we can construct the distinguished chiral extension $\INDL$ of $\LI$ with a lifting $\overline{i}:L\rightarrow \INDL$ of the canonical map $i:L\rightarrow \LI$.\\
\item From an extension $0\rightarrow \R\rightarrow L^c\rightarrow L\rightarrow 0$ we can construct a chiral extension $\INDLD$ with a map $L^c\rightarrow \INDLD$ given by the pull-back of $L\rightarrow \LI$.
\end{itemize}

\begin{remark}\label{rem2}
Clearly, if we have the extension $0\rightarrow \R\rightarrow L^c\rightarrow L\rightarrow 0$, we can also consider $L^c$ as a Lie$^*$ algebra acting on $\R$ 
via the projection $L^c\rightarrow \R$. In other words we forget about the extension and we only remember the Lie$^*$ algebra $L^c$. From point one of the above summary we can construct the distinguished chiral extension $\text{Ind}_{\R}(L^c)$ corresponding to this $L^c$ action on $\R$, together with a map
$\overline{i}:L^c\rightarrow \text{Ind}_{\R}(L^c)$.
\end{remark}

\subsection{}\label{fine}{ \scshape Different construction of $\INDLD$.}\label{bbb}
We will now explain a different construction of the chiral extension
 $$0\rightarrow \R\rightarrow \INDLD\rightarrow \LI\simeq\R\otimes L\rightarrow 0$$ that will be used later to construct $\OC$.\\
 
As it is explained in the Remark \ref{rem2}, given an $\R-$extension
\[
 0\rightarrow \R\xrightarrow{k}L^c\rightarrow L\rightarrow 0,
\]
we can consider the action of $L^c$ on $\R$ given by the projection $L^c\rightarrow L$ and construct
the distinguished  chiral extension $\INDLL$. This is a chiral $\R$-extension fitting into
\[
0\rightarrow \Delta_!(\R)\rightarrow \Delta_{!}(\INDLL)\rightarrow \Delta_!(\R\otimes L^c)\rightarrow 0,
\]
where
$\Delta_{!}(\INDLL)\simeq\bigslant{\Delta_!(\R)\oplus j_*j^*(\R\boxtimes L^c)}{\R\boxtimes L^c} .$ Since we ultimately want an extension of $\R$ by $\R\otimes L$, we have to quotient the above sequence by some additional relations. We will in fact
 obtain $\INDLD$ by 
taking the quotient of $\INDLL$ by the image of
 the difference of two maps from $\R\otimes \R\rightarrow \INDLL$. \\

\vspace{0.3 cm}
\noindent The above maps are given (under the Kashiwars's equivalence) by the following 
two maps from $\Delta_!(\R\otimes \R)$ to $\Delta_!(\INDLL)$.\\
1) The first map is given by the composition
\[
 \Delta_!(\R\otimes \R)\xrightarrow{m}\Delta_!(\R)\longhookrightarrow \Delta_{!}(\INDLL).
\]
\noindent 2) For the second map, consider the following commutative diagram:
\[
 \xymatrix{0\ar[r]&\R\boxtimes \R\ar[r]\ar[d]^{id\boxtimes k}& j_*j^*(\R\boxtimes \R)\ar[r]\ar[d]^{k\boxtimes id}&\Delta_!(\R\otimes \R)\ar[r]\ar[d]^{id\otimes k}&0\\
0\ar[r]&\R\boxtimes L^c\ar[r]\ar[d]& j_*j^*(\R\boxtimes L^c)\ar[r]\ar[d]^{\pi}&\Delta_!(\R\otimes L^c)\ar[r]\ar[d]&0\\
0\ar[r]&\Delta_!(\R)\ar[r]&\underset{\overset{\simeq}{\Delta_!(\INDLL)}}{\Delta_{!}(\R)\oplus j_*j^*(\R\boxtimes L^c)/\R\boxtimes L^c}\ar[r]&\Delta_!(\R\otimes L^c)\ar[r]& 0.
}
\]
We claim that the composition of the two vertical arrows in the middle (i.e. $\pi\circ (k\boxtimes id)$) factors through 
$\Delta_!(\R\otimes \R)$. In fact since the action of $L^c$ on $\R$ is given by the projection $L^c\rightarrow \R$,
the copy of $\R$ inside $L^c$ via k acts by zero. Hence the composition of the left most vertical arrows is zero, which shows that 
there is a well defined map $$\overline{k}:\Delta_!(\R\otimes \R)\rightarrow \Delta_!(\INDLL).$$
\noindent The quotient of $\INDLL$ by the image of the difference of the above maps is exactly $\INDLD$.

\begin{remark}
\noindent Note that the inclusion  $L^c\rightarrow \INDLD$ mentioned in the summary \ref{summary} corresponds to the composition
\begin{equation}\label{bbbb}
 \Delta_!(L^c)\xrightarrow{\Delta_!(\overline{i})}\Delta_!({\INDLL})\twoheadrightarrow \Delta_{!}(\INDLD).
\end{equation}
\end{remark}

\vspace{.2cm}
\subsection{}\label{particular}{\scshape A special case: deformations of $\R$.}
\noindent Now let $(\R,\,m:\R\otimes \R\rightarrow \R)$ be a commutative chiral algebra given as $\R:=\R_{\hbar}/\hbar\R_{\hbar}$,
where $\{\R_{\hbar}\}$ is a family of chiral algebras. Denote by $\{\,,\,\}$ the Poisson bracket on $\R$ defined as
\[
\{z,w\}=\frac{1}{\hbar}[z_{\hbar},w_{\hbar}]_{\hbar}\pmod \hbar, 
\]
where $z_{\hbar}=z\pmod \hbar$, $w_{\hbar}=w\pmod \hbar$ and $[\,,\,]_{\hbar}$ denotes the Lie$^*$ bracket on $\R_{\hbar}$ induced by the chiral product $\mu_{\hbar}$ restricted to $\R_{\hbar}\boxtimes \R_{\hbar}$.

\noindent Consider the quotient $\RC=\R_{\hbar}/\hbar^2\R_{\hbar}$. This is a Lie$^*$ algebra with bracket $[\,,\,]_c$ defined by
\[
[\overline{z_{\hbar}},\overline{w_{\hbar}}]_c=\overline{\frac{1}{\hbar}[z_{\hbar},w_{\hbar}]_h}.
\]
Consider the short exact sequence 
\begin{equation}\label{SS}
0\rightarrow \R\xrightarrow{\cdot \hbar}\RC\rightarrow \R\rightarrow 0, 
\end{equation}
\noindent and let us regard $\RC$ as a Lie$^*$ algebra acting on $\R$ via the projection $\RC\rightarrow \R$ followed by
 the Poisson bracket multiplied by\footnote{This correction is due to the fact that, as 
we saw in \ref{VVV}, the equivalence stated in 
Theorem \ref{bbdd} gives a quantization of $1/2\{\,,\,\}$.} $1/2$.
This sequence is an $\R$-extension of $\R$ in the sense we introduced in Definition \ref{lieextension}, therefore,
from what we have seen in \ref{bbb}, we can construct a chiral $\R$-extension of $\R\otimes \R$ by $\R$ (here $L=\R$ and $L^c=\RC$)
\begin{equation}\label{BB}
 0\rightarrow \R\rightarrow \INDR \rightarrow \R\otimes \R\rightarrow 0.
\end{equation}
Below we will use the above chiral extension to define the chiral algebroid $\Omega^c(\R).$
\subsection{}\label{construction}{\scshape The construction of $\OC$.}
We can now proceed to the construction of $\OC$. Recall that, because 
of the Poisson bracket on $\R$, the sheaf $\OO$ acquires a structure of a Lie$^*$ algebroid.
 In fact the action of $\R$ on $\R$ given by the Poisson bracket yields a Lie$^*$ $\R$-algebroid structure on $\R\otimes \R$. One checks 
easily that the kernel of the projection $\R\otimes \R\rightarrow \OO$, $z\otimes w\mapsto zdw$, is an ideal in $\R\otimes \R$, 
hence $\OO$ inherits a Lie$^*$ algebroid structure.\\

Recall that we denoted by $\mathcal{Q}^{ch}(\R)$ the groupoid of $\mathbbm{C}[\hbar]/\hbar^2$-deformations of our chiral-Poisson algebra $\R$, and that we want to understand how to construct the inverse to 
the functor
\[
 \mathcal{P}^{ch}(\OO)\rightarrow \mathcal{Q}^{ch}(\R),
\]
that assigns to a chiral extension $0\rightarrow \R\rightarrow \OC\rightarrow \OO\rightarrow 0$, its 
pull-back via the differential $d:\R\rightarrow \OO$. 

\subsection{}\label{rrr}
The inverse functor will be constructed as follows:  for any object in $\mathcal{Q}^{ch}(\R)$, i.e. to for any extension $0\rightarrow \R\xrightarrow{\cdot \hbar}\RC\rightarrow \R\rightarrow 0,$
we will consider the chiral extension
\[
  0\rightarrow \R\rightarrow \INDR \rightarrow \R\otimes \R\rightarrow 0
\]
described in the previous subsection.  We will quotient $\INDR$ by some additional relations
in order to impose the Leibniz rule on $\R\otimes \R$. These relations will be given, under Kashiwara's equivalence,
 as the image of 
a map from $\Delta_!(\RC\otimes \RC)$ to $\Delta_!(\INDR)$. More precisely, we will construct a map
from $j_*j^*(\RC\boxtimes \RC)$ to $\Delta_{!}(\IND)$ such that the composition with 
the projection $\Delta_{!}(\IND)\rightarrow \Delta_!(\INDR)$ vanishes when restricted to $\RC\boxtimes \RC$.
Hence it will induce a map $\Delta_!(\RC\otimes \RC)\rightarrow \Delta_!(\INDR).$ 
Form the sequence (\ref{BB}) we will  therefore obtain a chiral  $\R$-extension $\widetilde{\OC}$ of the Lie$^*$ algebroid $\OO$
\[
 0\rightarrow \R'\rightarrow \widetilde{\OC}\rightarrow \OO\rightarrow 0.
\]
We will then check that $\R'$, which a priori is a quotient of $\R$, is in fact $\R$ itself, and that the pull-back via the differential $d:\R\rightarrow \OO$ is the original sequence (\ref{SS}),
 with induced Poisson bracket given by $\{,\}$. This will imply that $\widetilde{\OC}$ is in fact the chiral extension $\OC$ given by Theorem \ref{bbdd}.\\ 

\subsection{}
The map from $j_*j^*(\RC\boxtimes \RC)$ to $\Delta_{!}(\IND)$ is defined as the sum of the following three maps:
\begin{enumerate}
\item The first map $\alpha_1$ is given by the composition $$j_*j^*(\RC\boxtimes \RC)\rightarrow j_*j^*(\R\boxtimes \RC)\rightarrow \Delta_!(\IND),$$ where the first map comes from the projection $\RC\rightarrow \R$.\\
\item The second map $\alpha_2$ is obtained from the first one by interchanging the roles of the factors in $j_*j^*(\RC\boxtimes \RC)$.\\
\item For the third map $\alpha_3$, note that the chiral bracket $\mu_{\hbar}$ on $\R_h$ gives rise to a map
\[
\cdot\mu_c:j_*j^*(\RC\boxtimes \RC)\rightarrow \Delta_{!}(\RC)
\]
and we compose it with the canonical map $\Delta_!(\RC)\rightarrow \Delta_!(\IND)$. 
\end{enumerate}
Now consider the linear combination $\alpha_1-\alpha_2-\alpha_3$ as a map from $j_*j^*(\RC\boxtimes \RC)$ to $\Delta_!(\IND)$. 
If we compose this map with the inclusion $\RC\boxtimes \RC\rightarrow j_*j^*(\RC\boxtimes \RC)$ and the projection onto $\INDR$, it is easy to see that the map vanishes. More precisely we have the following:
\begin{lem}
The composition
\[
\RC\boxtimes \RC\hookrightarrow j_*j^*(\RC\boxtimes \RC)\xrightarrow{\alpha_1-\alpha_2-\alpha_3}\Delta_!(\IND)\twoheadrightarrow\Delta_!(\INDR)
\]
vanishes. Thus it defines a map $Leib:\Delta_!(\RC\otimes \RC)\rightarrow \Delta_!(\INDR).$
\end{lem}
\begin{proof}
Since the action of $\RC$ on $\R$ is given by the projection $\RC\rightarrow \R$ and the Poisson bracket on $\R$ multiplied by $1/2$, and because of the 
relation $\sigma\circ\{\,,\,\}\circ \sigma=-\{\,,\,\}$, the maps $\alpha_1$ and $\alpha_2$ factor as
\[
 \xymatrix{\RC\boxtimes \RC\ar[r]\ar[dr]_{\alpha_1=\frac{1}{2}\{\,,\,\}=-\alpha_2}&\Delta_!(\IND)\ar[r]&\Delta_!(\INDR)\\
&\Delta_!(\R)\ar[ur]&}.
\]
Note that the above wouldn't have been true if we hadn't used the relation in $\IND$ as well. Moreover the third map, when composed with the projection to $\Delta_!(\INDR)$ is exactly 
\[
\RC\boxtimes \RC\rightarrow \R\boxtimes \R\xrightarrow{\{\,,\,\}} \Delta_!(\R)\rightarrow \Delta_!(\INDR),
\]
hence the combination $\alpha_1-\alpha_2-\alpha_3$ is indeed zero.
From the above we therefore get a map $\Delta_!(\RC\otimes \RC)\rightarrow \Delta_!(\INDR)$.\\
\end{proof}

We define $\widetilde{\OC}$ to be the quotient of $\INDR$ by 
the image of the corresponding map from $\RC\otimes \RC$ to $\INDR$ under the Kashiwara's equivalence.\\

\begin{remark}\label{rr}
 Note that the map $\RC\otimes \RC\rightarrow \INDR$ indeed factors through $\RC\otimes \RC\twoheadrightarrow \R\otimes \R$. To show this it is enough to show that 
the map $j_*j^*(\RC\boxtimes \RC)\rightarrow \Delta_!(\INDR)$ factors through $j_*j^*(\RC\boxtimes \RC)\rightarrow j_*j^*(\R\boxtimes \R)$. If so, then the diagram below would  imply that the composition 
$\R\boxtimes \R\rightarrow \Delta_!(\INDR)$ is zero, and we are done:
\[
 \xymatrix{0\ar[r]&\RC\boxtimes \RC\ar[r]\ar[d]&j_*j^*(\RC\boxtimes \RC)\ar[r]\ar[d]&\Delta_!(\RC\otimes \RC)\ar[d]\ar[r]&0\\
0\ar[r]&\R\boxtimes \R\ar[r]&j_*j^*(\R\boxtimes \R)\ar[r]\ar[d]&\Delta_!(\R\otimes \R)\ar[r]&0\\
&&\Delta_!(\INDR)&&}.
\]
To show that the map factors as 
\[
 \xymatrix{j_*j^*(\RC\boxtimes \RC)\ar[r]\ar[d]&\Delta_!(\INDR)\\
j_*j^*(\R\boxtimes \R)\ar[ur]&}
\]
we need to show that the composition of the map $\alpha_1-\alpha_2-\alpha_3$ with the two embeddings 
$j_*j^*(\RC\boxtimes \R)\longhookrightarrow j_*j^*(\RC\boxtimes \RC)$ and  
$j_*j^*(\R\boxtimes \RC)\longhookrightarrow j_*j^*(\RC\boxtimes \RC)$ is zero.
We'll do only one of them (the second one can be done similarly). For the first embedding the map $\alpha_2$ is zero 
(since we are projecting the second $\RC$ onto $\R$) whereas the first map (because of the relations in $\INDR$) is equal to minus the composition
\[
 j_*j^*(\RC\boxtimes \R)\rightarrow j_*j^*(\R\boxtimes \R)\xrightarrow{\mu}\Delta_!(\R)\rightarrow \Delta_!(\INDR)
\]
which is exactly the third map when restricted to $j_*j^*(\RC\boxtimes \R)$.
\end{remark}

\subsection{} \label{cacca}Recall that we defined $\widetilde{\OC}$ to be the quotient 
of $\INDR$ by the image of the map $Leib$ obtained using the combination $\alpha_1$-$\alpha_2$-$\alpha_3$. 
By construction we have a short exact sequence 
\begin{equation}\label{omega}
0\rightarrow \R'\rightarrow \widetilde{\OC}\rightarrow \OO\rightarrow 0,
\end{equation}
where $\R'$ is a certain quotient of $\R$. In the rest of this section we will show that the above extension is in fact isomorphic to the extension of $\Omega^1(\R)$ given in Theorem \ref{bbdd}.  This is equivalent to the following:
\begin{prop}
Consider the extension of $\Omega^1(\R)$ given by (\ref{omega}). Then we have 
\begin{enumerate}
 \item $\R'=\R$.\\
\item The pull-back of (\ref{omega}) via the differential $d:\R\rightarrow \OO$ is the original sequence (\ref{SS}).
\end{enumerate}
\end{prop}
\begin{proof}
\noindent To show that $\R'=\R$, consider the chiral extension given by the equivalence of Theorem \ref{bbdd}. This is an extension of $\OO$ such that the pull back via the differential $\R\rightarrow \OO$ is the sequence (\ref{SS}). that is, we have the following diagram
\begin{equation}\label{dc}
\xymatrix{0\ar[r]& \R\ar[r]^-{i}&\OC\ar[r]&\OO\ar[r]&0,\\
0\ar[r]&\R\ar[r]\ar[u]&\RC\ar[r]^{\pi}\ar[u]^{d^c}&\R\ar[r]\ar[u]^d&0}
\end{equation}
with $d^c$ a derivation, i.e. as maps from $j_*j^*(\RC\boxtimes \RC)$ to $\Delta_!(\OC)$, we have $d^c(\mu_{c})=\mu_{\R,\OC}(\pi,d^c)-\sigma\circ \mu_{\R,\OC}\circ \sigma(d^c, \pi)$,
 where $\mu_{c}$ is the chiral product on $\RC$ and $\mu_{\R,\OC}$ is the chiral action of $\R$ on $\OC$.
We claim that there is a map of short exact sequences
\[
\xymatrix{0\ar[r]&\Delta_!(\R)\ar[r]\ar[d]^-{id}&\Delta_!(\INDR)\ar[r]\ar[d]&\Delta_!(\R\otimes \R)\ar[r]\ar[d]&0\\
0\ar[r]&\Delta_!(\R)\ar[r]&\Delta_!(\OC)\ar[r]&\Delta_!(\OO)\ar[r]&0
}
\]
that factors through $0\rightarrow \Delta_!(\R')\rightarrow \Delta_!(\widetilde{\OC})\rightarrow \Delta_!(\OO)\rightarrow 0$, 
and moreover induces an isomorphism from $\OO$ to $\OO$. This would imply that $\R'$, which a 
priori is a quotient of $\R$, is in fact $\R$ itself. Furthermore, the fact that it is an isomorphism on $\OO$, would also imply that $\widetilde{\OC}\simeq \OC$, hence the pull-back via $d:\R\rightarrow \OO$ would indeed be the original sequence $0\rightarrow \R\rightarrow \RC\rightarrow \R\rightarrow 0$.\\

\vspace{0.2cm}
\noindent To prove the claim, consider the map $d^c:\RC\rightarrow \OC$ given by (\ref{dc}). Using the chiral $\R$-module structure $\mu_{\R,\OC}$ on $\OC$, we can consider the composition
\[
j_*j^*(\R\boxtimes \RC)\xrightarrow{i\boxtimes d^c}j_*j^*(\R\boxtimes \OC)\xrightarrow{\mu_{\R,\OC}}\Delta_!(\OC).
\]
The above composition can be extended to a map from $\Delta_!(\R)\oplus j_*j^*(\R\boxtimes \RC)\rightarrow \Delta_!(\OC)$, by setting the map to be $\Delta_!(i)$ on $\Delta_!(\R)$. It is straightforward to check that this map factors through a map $D^c$
\[
D^c: \Delta_!(\INDR)\rightarrow \Delta_!(\OC).
\] 
\noindent Note that, by construction, the resulting map $\overline{D}^c: \R\otimes \R\rightarrow \OO$ is the one given by $z\otimes w\mapsto zdw$, for $z$ and $w$ 
in $\R$, and that the kernel of this map is just the ideal defining the Leibniz rule.\\
To show that $D^c$ factors through $0\rightarrow \Delta_!(\R')\rightarrow \Delta_!(\widetilde{\OC})\rightarrow \Delta_!(\OO)\rightarrow 0$,  
we need to show that the
 composition of $D^c$ with the map $$Leib :\Delta_!(\RC\otimes \RC)\rightarrow \Delta_!(\INDR)$$ given in \ref{rrr}. 
vanishes.
Hence we are left with checking that the composition
\[
\Delta_!(\RC\otimes \RC)\xrightarrow{Leib}\Delta_!(\INDR)\xrightarrow{\Delta_!(D^c)}\Delta_!(\OC)
\] is zero. For this,  recall that the map $Leib$ was constructed using the linear combination $\alpha_1-\alpha_2-\alpha_3$ of three 
maps $\alpha_1$, $\alpha_2$ and $\alpha_3$ from $j_*j^*(\RC\boxtimes \RC)$. By looking at the map
\[
j_*j^*(\RC\boxtimes \RC)\xrightarrow{\alpha_1-\alpha_2-\alpha_3}\Delta_!(\INDR)\xrightarrow{\Delta_!(D^c)}\Delta_!(\OC),
\]
we see that the condition on $d^c$ being a derivation, implies that the above composition vanishes. Indeed $\Delta_!(D^c)\circ \alpha_1$ is given by
\[
 j_*j^*(\RC\boxtimes \RC)\xrightarrow{\pi\boxtimes id}j_*j^*(\R\boxtimes \RC)\xrightarrow{id\boxtimes d^c}\j_*j^*(\R\boxtimes \OC)\xrightarrow{\mu_{\R,\OC}}\Delta_!(\OC).
\]
The map $\Delta_!(D^c)\circ\alpha_2$ is given by the above by applying the transposition of variables $\sigma$, whereas the third map
\[j_*j^*(\RC\boxtimes \RC)\xrightarrow{\mu_{c}}\Delta_!(\RC)\rightarrow \Delta_!(\INDR)\xrightarrow{D^c}\Delta_!(\OC)\]
is equal to
 $j_*j^*(\RC\boxtimes \RC)\xrightarrow{\mu_{c}}\Delta_!(\RC)\xrightarrow{\Delta_!(d^c)}\Delta_!(\OC).$ 
Therefore the above maps coincide with the terms in the relation $d^c(\mu_{c})=\mu_{\R,\OC}(\pi,d^c)-\sigma\circ \mu_{\R,\OC}\circ \sigma(d^c, \pi)$, and hence 
$\Delta_!(D^c)\circ Leib$ in zero.
Note that the resulting map  
$$\Delta_!(\INDR)\overset{\pi}{\twoheadrightarrow}\Delta_!(\R\otimes \R)\xrightarrow{\Delta_!(\overline{D}^c)}\Delta_!(\OO)$$ induces an isomorphism 
\[
 \Delta_!(\OO)\simeq\Delta_!(\R\otimes \R)/\text{Im}(\pi\circ Leib)\xrightarrow{\sim}\Delta_!(\OO).
\]

\noindent This conclude the proof of the proposition.
\end{proof}


\section{Construction of the map F}\label{mapF}
\subsection{}\label{prope} Recall that, by definition, Theorem \ref{main2} amounts to the  construction of a map of Lie$^*$
algebras $F:\Omega^c(\ZZ)\rightarrow \B$ compatible with the $\ZZ$ structure on both sides and such that 
\begin{enumerate}
\item $F$ restricts to the embedding $l$ (given in Remark \ref{embedding}) 
on $\ZZ$.\\
\item The following diagram commutes:
\[
 \xymatrix{\Omega^c(\ZZ)/\ZZ\ar[rr]^{F}&&(\B)_1/\ZZ.\\
&\Omega^1(\ZZ)\ar[ul]^{\simeq}\ar[ur]_{\simeq}&}
\]
\end{enumerate}

\subsection{}\label{well defined} 
Since $\Delta_!(\Omega^c(\ZZ))$ was constructed as a quotient of $\Delta_!(\INDZ)$ and 
 since, by definition, $$\Delta_!(\INDZ)=\Delta_!(\ZZ)\oplus j_*j^*(\ZZ\boxtimes\ZZ^c)/\ZZ\boxtimes \ZZ^c,$$
 to construct any map $F$ from 
$\Omega^c(\ZZ)$ to $\B$
we can proceed as follows:
\begin{itemize}
 \item first we construct a map 
 $f:\ZZ^c\rightarrow \B$.
\item  Using the chiral bracket $\mu'$ on $\B$ we consider the composition
\[
j_*j^*(\ZZ\boxtimes \ZZ^c)\xrightarrow{l\boxtimes f} j_*j^*(\B\boxtimes \B)\xrightarrow{\mu'}\Delta_!(\B).
\]
This composition yields a map $$\hat{F}:\Delta_!(\ZZ)\oplus j_*j^*(\ZZ\boxtimes \ZZ^c) \rightarrow \Delta_!(\B),$$ 
by sending $\Delta_!(\ZZ)$ to $\B$ via $\Delta_!(l)$.
\item We check that the above map factors through 
a map $$\tilde{F}:\Delta_!(\INDZ)\rightarrow \Delta_!(\B).$$
\item We check that in fact if factors through $\overline{F}:\Delta_!(\INDZC)\rightarrow\Delta_!(\B)$.
\item We verify that the relations defining $\Delta_!(\Omega^c(\ZZ))$ as a quotient of $\Delta_!(\INDZC)$ are satisfied,
 i.e. that $\overline{F}$ gives the desired  map $F$ from $\Omega(\ZZ)^c$ to $\B$ under the Kashiwara equivalence.
\end{itemize}
\begin{remark}\label{condition1}
Note 
that any map $F$ constructed as before, automatically satisfies the first condition in \ref{prope}, hence to prove Theorem \ref{main2}, once the map $f$ is defined, we only have to verify that condition $(2)$ in \ref{prope} is satisfied, i.e. that the diagram above commutes. 
\end{remark}
\subsection{}{\scshape Definition of the map $f$. }
We will now define the map $f:\ZZ^c\rightarrow \B$ and hence, according to the first two points in \ref{well defined}, 
the map $\hat{F}:\Delta_!(\ZZ)\oplus j_*j^*(\ZZ\boxtimes \ZZ^c) \rightarrow \Delta_!(\B)$. In \ref{end}, assuming that it factors 
through a map
$F:\Omega^c(\ZZ)\rightarrow \B$, we will then show that it satisfies the second condition in \ref{prope}. This will conclude the proof of 
Theorem \ref{main2}. The poof that 
it factors through $\Omega^c(\ZZ)$ (which amounts to the proof of the remaining last
three points in \ref{well defined}) will be postponed until \ref{relationshold}.
\subsection{}\label{mapf}
To define the map $f:\ZZ^c\rightarrow \B$ we will use the following three facts:
\begin{enumerate}
 \item There exist two embeddings 
 \[
\ZZ\xrightarrow{l}\B\xleftarrow{r}\ZZ
\]
constructed by applying the functor 
  $\Psi$ to the two embeddings in (\ref{EM}). In fact, by doing it, we obtain two maps
\[
\mathcal{W}_{\hbar}\xrightarrow{l_{\hbar}}(\Psi\boxtimes \Psi)(\mathcal{D}_{\hbar})\xleftarrow{r_{\hbar}}\mathcal{W}_{-\hbar}
\]
such that $l:=l_0=r_0\circ \eta=:r\circ \eta$, where we are denoting by $l_{\hbar}$ and $r_{\hbar}$ the maps $\Psi(l_{\hbar})$ and $\Psi(r_{\hbar})$ respectively. The two embedding of $\ZZ$ correspond to the above maps when $\hbar=0$. \\
\item There is a well defined map $$e:\mathcal{W}_{\hbar}\rightarrow \mathcal{W}_{-\hbar}.$$ In fact, 
since $\mathcal{W}_{\hbar}=\Psi(\mathcal{A}_{\hbar})$, and since $\mathcal{A}_{-\hbar}$ is isomorphic to $\mathcal{A}_{\hbar}$ as vector space
with the action of $\mathbbm{C}[\hbar]$ modified to $\hbar\cdot a=-\hbar a$, $a\in \mathcal{A}_{-\hbar}$, we can consider the 
map $\mathcal{W}_{\hbar}\rightarrow \mathcal{W}_{-\hbar}$ that simply sends
$\hbar$ to $-\hbar$. \\
\item The involution 
$\eta:\ZZ\rightarrow \ZZ$ can be extended to a map
$\eta:\mathcal{W}_{\hbar}\rightarrow \mathcal{W}_{\hbar}$ 
by setting $\eta(h)=h$.
\end{enumerate}
We define $f$ in the following way: for every $z_{\hbar}\in \ZZ^c=\mathcal{W}_{\hbar}/\hbar^2\mathcal{W}_{\hbar}$ we set 
\[
f(z_{\hbar})=\frac{1}{2}\frac{l_{\hbar}(z_{\hbar})-r_{\hbar}(\eta(e(z_{\hbar})))}{\hbar}\pmod \hbar.
\]
This is a well defined element in $\B$ because $l_0=r_0\circ \tau$, i.e. the numerator vanishes mod $\hbar$.\\

Assuming the proposition below, we will now show
that the resulting $F$ satisfies condition $(2)$ of \ref{prope}, which, according to Remark \ref{condition1}, concludes the proof of Theorem \ref{main2}. Proposition \ref{hold} will be proved later in \ref{relationshold}.
\begin{prop}\label{hold}
 The map  $$\hat{F}:\Delta_!(\ZZ)\oplus j_*j^*(\ZZ\boxtimes \ZZ^c) \rightarrow \Delta_!(\B),$$ obtained by using $f:\ZZ^c\rightarrow \B$ from above,
factors through a map $F:\Delta_!(\Omega^c(\ZZ))\rightarrow \Delta_!(\B)$.
\end{prop}

\subsection{}\label{end}{\scshape End of the proof of Theorem \ref{main2}.} 
\begin{proof}We are now ready to finish the proof of Theorem \ref{main2}, 
which, according to Remark \ref{condition1}, amounts 
to check that 
\[
 \xymatrix{\Omega^c(\ZZ)/\ZZ\ar[rr]^{F}&&(\B)_1/\ZZ\\
&\Omega^1(\ZZ)\ar[ul]^{\simeq}\ar[ur]_{\simeq}&}
\]
commutes. In order to do so, we will show that it commutes when composed with the map $d:\ZZ\rightarrow \Omega^1(\ZZ)$.
By looking at the composition 
$$\ZZ\xrightarrow{d}\Omega^1(\ZZ)\rightarrow\Omega^c(\ZZ)/\ZZ\xrightarrow{F} (\B)_{1}/\ZZ,$$ we see that, for $z\in \ZZ$,
the resulting map is 
\begin{equation}\label{altra}
z\mapsto \frac{1}{2}\frac{l_{\hbar}(z_{\hbar})-r_{\hbar}(\eta(e(z_{\hbar})))}{\hbar}\pmod \hbar,
\end{equation}

\noindent where $z_{\hbar}$ is any lifting of $z$ to $\ZZ^c$. Note that this map is well defined only after taking the quotient of $\B$ by $\ZZ$.\\
For the other composition, we first need to recall how the isomorphism $\Omega^1(\ZZ)\xrightarrow{\sim} (\B)_1/\ZZ$ was constructed. 
Recall from \ref{filtration1} that the filtration on $\B$ is the one induced (under $\Psi\boxtimes \Psi$) from the 
isomorphism $G$ given in Theorem \ref{REN}. Therefore the isomorphism above is the one corresponding to the composition
\begin{eqnarray*}
( \A_{crit}\underset{\ZZ}{\otimes}\A_{crit})\underset{\ZZ}{\otimes} \Omega^1(\ZZ)&\xrightarrow{\sim}U(\A^{ren,\tau})_1/ (\A_{crit}\underset{\ZZ}{\otimes}\A_{crit})\xrightarrow{G}\\
&\rightarrow \mathcal{D}^0_{crit}/l(\A_{crit})+r(\A_{crit})
\end{eqnarray*}
under $(\Psi\boxtimes \Psi)$ (here, for simplicity,  we are 
denoting the chiral envelope $U((\AL\underset{\ZZ}{\otimes}\AL),\,\A^{ren,\tau})$ by $U(\A^{ren,\tau})$). If we
 consider the inclusion of $\Omega^1(\ZZ)$ followed by
the first arrow from above, it is clear that the image in 
$U(\A^{ren,\tau})/( \A_{crit}\underset{\ZZ}{\otimes}\A_{crit})$ is $G[[t]]\times G[[t]]$ invariant.
 In particular it means that
the image of $\Omega^1(\ZZ)$ maps to $(\Psi\boxtimes \Psi)(U(\A^{ren,\tau}))/\ZZ$. 
Now, by looking at the definition of the map $G$ (see \cite{FG} 5.5.), we see that the the map 
\[
 \ZZ\xrightarrow{d} \Omega^1(\ZZ)\rightarrow (\Psi\boxtimes \Psi)(U(\A^{ren,\tau}))_1/\ZZ\xrightarrow{(\Psi\boxtimes \Psi)(G)}(\B)_1/\ZZ ,
\]
\noindent is indeed given by (\ref{altra}). This completes the proof of Theorem \ref{main}.
\end{proof}
We will now give the proof of Proposition \ref{hold}, which will occupy the rest of the article.
\subsection{}\label{relationshold}{\scshape Proof of Proposition \ref{hold}.}
\begin{proof}
Recall that the proof of Proposition \ref{hold} consists in showing the following:
\begin{enumerate}
\item the map $\hat{F}: \Delta_!(\ZZ)\oplus j_*j^*(\ZZ\boxtimes \ZZ^c) \rightarrow \Delta_!(\B)$ factors through 
a map $$\tilde{F}:\Delta_!(\INDZ)\rightarrow \Delta_!(\B).$$
\item the map $\tilde{F}$ factors through $\overline{F}:\Delta_!(\INDZC)\rightarrow\Delta_!(\B)$.
\item The relations defining $\Delta_!(\Omega^c(\ZZ))$ as a quotient of $\Delta_!(\INDZC)$ are satisfied,
 i.e. $\overline{F}$ gives the desired  map $F$ from $\Omega(\ZZ)^c$ to $\B$ under the Kashiwara equivalence.
\end{enumerate}
For this we will need the following Lemma.
\begin{lem}\label{zerobracket}
  The composition 
\[
 \mathcal{W}_{\hbar}\boxtimes  \mathcal{W}_{-\hbar}\xrightarrow{l_{\hbar}\boxtimes r_{\hbar}} (\Psi\boxtimes \Psi)(\mathcal{D}_{\hbar})\boxtimes (\Psi\boxtimes \Psi)(\mathcal{D}_{\hbar})\xrightarrow{\mu'} \Delta_!((\Psi\boxtimes \Psi)(\mathcal{D}_{\hbar}))  
\] is zero.
\end{lem}
\begin{proof}
In \cite{FG} Lemma 5.2 it is shown that the composition 
\[
 \mathcal{A}_{\hbar}\boxtimes  \mathcal{A}_{-\hbar}\xrightarrow{l_{\hbar}\boxtimes r_{\hbar}} \mathcal{D}_{\hbar}\boxtimes \mathcal{D}_{\hbar}\xrightarrow{\mu'} \Delta_!(\mathcal{D}_{\hbar})  
\] is zero. In other words the two embeddings centralize each other. The Lemma then, immediately 
follows by applying the functor $(\Psi\boxtimes \Psi)$.

\end{proof}

\subsection{}{\scshape Proof of $(1).$} To prove that $\hat{F}$ factors through $$\overline{F}:\Delta_!(\INDZC)\rightarrow\Delta_!(\B)$$ we use Lemma \ref{zerobracket}. 
Recall that we defined $f$ from  $\ZZ^c$ to $\B$ to be 
\[
f(z_{\hbar})=\frac{1}{2}\frac{l_{\hbar}(z_{\hbar})-r_{\hbar}(\eta(e(z_{\hbar})))}{\hbar}\pmod \hbar.
\]
Because of the above Lemma, it is clear that, when we consider the inclusion $\ZZ\boxtimes \ZZ^c\hookrightarrow j_*j^*(\ZZ\boxtimes \ZZ^c)$ 
and the composition with the map to $\Delta_!(\B)$, the resulting map factors as:
\[
 \xymatrix{\ZZ\boxtimes \ZZ^c\ar[r]\ar[d]& j_*j^*(\ZZ\boxtimes \ZZ^c)\ar[r]&\Delta_!(\B)\\
\ZZ\boxtimes \ZZ\ar[r]^{\frac{1}{2}\{\,,\,\}}& \Delta_!(\ZZ)\ar[ru]_{\Delta_!(l)}&},
\]
which implies that the map factors through a map  $\tilde{F}:\Delta_!(\INDZ)\rightarrow \Delta_!(\B)$.
\begin{remark}\label{rem9}
Note that when we restrict the map $f:\ZZ^c \rightarrow \B$ to $\ZZ\xrightarrow{\cdot \hbar}\ZZ^c$, because of the flip from $\hbar$ to $-\hbar$ in the definition of $e$,
 we simply obtain the inclusion $\ZZ\xrightarrow{l}\B$.
\end{remark}
\subsection{}{\scshape Proof of $(2).$}
Now we want to check that the relations
 defining $\Delta_!(\INDZC)$ as a quotient of $\Delta_!(\INDZ)$ are satisfied, i.e. that $\tilde{F}$ factors through a map $\overline{F}:\Delta_!(\INDZC)\rightarrow \Delta_!(\B)$.

\noindent First of all, recall that to pass from $\Delta_!(\INDZ)$ to $\Delta_!(\INDZC)$ we took the quotient by the image of the difference of two maps from $\Delta_!(\ZZ\otimes \ZZ)$ to $\Delta_!(\INDZ)$. The first map was given by 
\begin{equation}\label{M1}
\Delta_!(\ZZ\otimes \ZZ)\xrightarrow{m}\Delta_!(\ZZ)\rightarrow \Delta_!(\INDZ),
\end{equation}
while the second map was induced by the 
composition $$j_*j^*(\ZZ\boxtimes \ZZ)\xrightarrow{id\boxtimes \cdot h} j_*j^*(\ZZ\boxtimes \ZZ^c)\twoheadrightarrow \Delta_!(\INDZ),$$ 
which vanishes on $\ZZ\boxtimes \ZZ\hookrightarrow j_*j^*(\ZZ\boxtimes \ZZ)$.
When we compose the map (\ref{M1}) with $\tilde{F}$, we get $l\circ m=\mu'\circ (l\boxtimes l)$. 
However when we compose the second map with 
\[
j_*j^*(\ZZ\boxtimes \ZZ^c)\xrightarrow{l\boxtimes f}j_*j^*(\B\boxtimes \B)\xrightarrow{\mu'}\Delta_!(\B),
\]
because of Remark \ref{rem9}, we see that this map corresponds to $\mu'\circ(l\boxtimes l)$ hence the
 difference of the images goes to zero under $\tilde{F}$.

\subsection{}{\scshape Proof of $(3).$}
Now we are left with checking that $\tilde{F}$ factors through $$\overline{F}:\Delta_!(\INDZC)\rightarrow \Delta_!(\Omega^c(\ZZ)).$$
This will occupy the rest of the article. Recall that $\Delta_!(\Omega^c(\R))$ was given as a quotient of $\Delta_!(\INDZC)$ by the map $Leib$. The Leibniz relation was given as the image of a map 
\[
\Delta_!(\ZZ^c\otimes \ZZ^c)\rightarrow \Delta_!(\INDZC)
\]
and this map was the sum of three maps, $\alpha_1,$ $\alpha_2$ and $\alpha_3$,  from $j_*j^*(\ZZ^c\boxtimes \ZZ^c)$ which vanished on $\ZZ^c\boxtimes \ZZ^c$. Hence we want to check that the composition   $$\Delta_!(\ZZ^c\otimes \ZZ^c)\xrightarrow{Leib}\Delta_!(\INDZC)\xrightarrow{\overline{F}}\Delta_!(\B)$$
vanishes.
Instead of considering the map from $\Delta_!(\ZZ^c\otimes \ZZ^c)$ we can consider the three maps 
\begin{equation}\label{fine?}
\xymatrix{
j_*j^*(\ZZ^c\boxtimes \ZZ^c) \ar@<2 ex>[r]^{\alpha_1}\ar@<-2ex>[r]^{\alpha_3}\ar[r]^{\alpha_2}& \Delta_{!}(\INDZC)\ar[r]^-{\overline{F}}&\Delta_!(\B),}
\end{equation}
and show that the composition $\overline{F}\circ(\alpha_1-\alpha_2-\alpha_3)$ is zero.
Recall that the first map, $\alpha_1$, was given by projecting onto $j_*j^*(\ZZ\boxtimes \ZZ^c)$, 
i.e.
\[
\xymatrix{
 j_*j^*(\ZZ^c\boxtimes \ZZ^c)\ar@/^1pc/ [rr]^{\alpha_1}\ar[r]& j_*j^*(\ZZ\boxtimes \ZZ^c)\ar[r]_{\beta}&\Delta_!(\INDZC)},
\]
where $\beta$ denotes the second component of the projection
\[
\Delta_!(\ZZ)\oplus j_*j^*(\ZZ\boxtimes \ZZ^c)\twoheadrightarrow \Delta_!(\INDZC).
\]
The second map, $\alpha_2$, was given by $\sigma\circ\alpha_1\circ \sigma$, and
the third map $\alpha_3$ was given by the composition
\begin{eqnarray*}
 j_*j^*(\ZZ^c\boxtimes \ZZ^c)\xrightarrow{\mu_c}\Delta_!(\ZZ^c)\xrightarrow{\overline{i}}\Delta_!(\INDZ)\twoheadrightarrow\\
\twoheadrightarrow\Delta_!(\INDZC).
\end{eqnarray*}

\noindent When we compose $\alpha_3$ with the map $\overline{F}:\Delta_!(\INDZC)\rightarrow \Delta_!(\B)$, it is easy to 
see that the unit axiom implies 
that the composition is equal to
\begin{equation}
j_*j^*(\ZZ^c\boxtimes \ZZ^c)\xrightarrow{\mu_c}\Delta_!(\ZZ^c)\xrightarrow{\Delta_!(f)}\Delta_{!}(\B).
\end{equation}


\noindent Now consider the chiral algebra $(\Psi\boxtimes \Psi)(\mathcal{D}_{\hbar})$ and denote by $\mu'_{\hbar}$ its chiral operation. 
Consider the map 
$$f_{\hbar}:\mathcal{W}_{\hbar}\rightarrow (\Psi\boxtimes \Psi)(\mathcal{D}_{\hbar})$$ 
\[f_{\hbar}(z_{\hbar})= \frac{1}{2}\frac{l_{\hbar}(z_{\hbar})-r_{\hbar}(\eta(e(z_{\hbar})))}{\hbar}\in (\Psi\boxtimes \Psi)(\mathcal{D}_{g,h}),\]
(i.e. we are not taking this element $\pmod \hbar$). 
It is clear that the three maps $\alpha_1''$, $\alpha_2''$ and $\alpha_3''$ given by 
\[
j_*j^*(\WW\boxtimes \WW)\xrightarrow{l_{\hbar}\boxtimes f_{\hbar}}j_*j^*((\Psi\boxtimes \Psi)(\mathcal{D}_{\hbar})\boxtimes (\Psi\boxtimes \Psi)(\mathcal{D}_{\hbar}))\xrightarrow{\mu'_{\hbar}}
\]
\begin{equation}
\tag{$\alpha_1''$}
 \xrightarrow{\mu'_{\hbar}}\Delta_!((\Psi\boxtimes \Psi)(\mathcal{D}_{\hbar}))\rightarrow \Delta_!(\B),
\end{equation}

\[
j_*j^*(\WW\boxtimes \WW)\xrightarrow{(r_{\hbar}\circ\eta\circ e\boxtimes f_{\hbar})\circ\sigma}j_*j^*((\Psi\boxtimes \Psi)(\mathcal{D}_{\hbar})\boxtimes (\Psi\boxtimes \Psi)(\mathcal{D}_{\hbar})\xrightarrow{\sigma\circ\mu'_{\hbar}}
\]
\begin{equation}
\tag{$\alpha_2''$}
 \xrightarrow{\sigma\circ\mu'_{\hbar}}\Delta_!((\Psi\boxtimes \Psi)(\mathcal{D}_{\hbar}))\rightarrow \Delta_!(\B),
\end{equation}

\begin{equation}
\tag{$\alpha_3''$}
j_*j^*(\WW\boxtimes \WW)\xrightarrow{\mu_{\hbar}}\Delta_!(\WW)\xrightarrow{\Delta_!(f_{\hbar})}\Delta_!((\Psi\boxtimes \Psi)(\mathcal{D}_{\hbar}))\rightarrow \Delta_!(\B),
\end{equation}
\noindent respectively, vanish on $j_*j^*(\hbar^2(\WW\boxtimes\WW))\hookrightarrow j_*j^*(\WW\boxtimes \WW)$, in particular
 they define well defined maps from $j_*j^*(\ZZ^c\boxtimes \ZZ^c)$ to $\Delta_!(\B)$.
Moreover the resulting maps coincide with $\alpha_1$, $\alpha_2$ and $\alpha_3$ composed with $\overline{F}$. In fact, the first and the last coincide by definition. 
For the second one, simply note that, modulo $\hbar$, the map $r_{\hbar}\circ \eta\circ e$ equals $l$.\\

\vspace{0.3cm}
By the above, to show that the combination of the three maps given in (\ref{fine?}) is zero, it is enough to check that the combination 
$\alpha_1''-\alpha_2''-\alpha_3''$ of the above three maps vanishes.\\
Let us denote by $\alpha_1'$, $\alpha_2'$ and $\alpha_3'$ the maps from $j_*j^*(\WW\boxtimes \WW)$ to $\Delta_!((\Psi\boxtimes \Psi)(\mathcal{D}_{\hbar}))$
corresponding to $\alpha_1'',$ $\alpha_2''$ and $\alpha_3''$ respectively (i.e. before taking the maps $\pmod \hbar$).
We will show that the combination $\alpha_1'-\alpha_2'-\alpha_3'$ is already zero.\\
Because $(\Psi\boxtimes \Psi)(\mathcal{D}_{\hbar})$ 
is $h$-torsion free, it is enough to show that the three maps agree after multiplication by $h$. But now note that each of the maps 
\[
 h\alpha_1',\,h\alpha_2',\,h\alpha_3'\in \Hom(j_*j^*(\WW\boxtimes \WW), \Delta_!((\Psi\boxtimes \Psi)(\mathcal{D}_{\hbar}))),
\]
is the sum of two terms, and the sum of the resulting six maps is zero.
Indeed $h\alpha_1'$ equals the sum of the following two maps:
\begin{equation}\label{term1}
j_*j^*(\WW\boxtimes\WW)\xrightarrow{l_{\hbar}\boxtimes \l_{\hbar}} j_*j^*((\Psi\boxtimes \Psi)(\mathcal{D}_{\hbar}))\xrightarrow{\mu_{\hbar}'}\Delta_!(\mathcal{D}_{\hbar}),
\end{equation}
\begin{equation}\label{term2}
 j_*j^*(\WW\boxtimes \WW)\xrightarrow{l_{\hbar}\boxtimes r_{\hbar}\circ \eta\circ e}j_*j^*((\Psi\boxtimes \Psi)(\mathcal{D}_{\hbar}))\xrightarrow{\mu_{\hbar}'}\Delta_!(\mathcal{D}_{\hbar}).
\end{equation}
On the other hand, the map $h\alpha_3'$ is given by the sum of the following
\begin{equation}\label{term11}
 j_*j^*(\WW\boxtimes\WW)\xrightarrow{\mu_{\hbar}}\Delta_!(\WW)\xrightarrow{\Delta_!(l_{\hbar})}\Delta_!(\mathcal{D}_{\hbar}),
\end{equation}
\begin{equation}\label{term22}
  j_*j^*(\WW\boxtimes\WW)\xrightarrow{\eta\circ\mu_{\hbar}\circ(e\boxtimes e)}\delta_!(\mathcal{W}_{-\hbar})\xrightarrow{\Delta_!(r_{\hbar})}\Delta_!(\mathcal{D}_{\hbar}).
\end{equation}
It is clear that the map (\ref{term1}) equals minus the map (\ref{term11}). Similarly, the relation
$\mu'_{\hbar}=-\sigma\circ \mu'_{\hbar}\circ \sigma$ guarantees that the two maps summing 
up to $h\alpha_2'$, given by
\[
j_*j^*(\WW\boxtimes\WW)\xrightarrow{(r_{\hbar}\circ\eta\circ e\boxtimes l_{\hbar})\circ\sigma} j_*j^*((\Psi\boxtimes \Psi)(\mathcal{D}_{\hbar}))\xrightarrow{\sigma\circ\mu_{\hbar}'}\Delta_!(\mathcal{D}_{\hbar}), 
\]
\noindent and
\[
j_*j^*(\WW\boxtimes\WW)\xrightarrow{(r_{\hbar}\circ\eta\circ e\boxtimes r_{\hbar}\circ\eta\circ e)\circ\sigma} j_*j^*((\Psi\boxtimes \Psi)(\mathcal{D}_{\hbar}))\xrightarrow{\sigma\circ\mu_{\hbar}'}\Delta_!(\mathcal{D}_{\hbar}),
\]
cancel with the remaining maps (\ref{term2}) and (\ref{term22}) respectively.\\

\noindent Hence the composition $\alpha_1-\alpha_2-\alpha_3$ as a map from $j_*j^*(\ZZ^c\boxtimes \ZZ^c)$ to 
$\Delta_{!}(\B)$ is zero, i.e. the map $\overline{F}:\Delta_!(\INDZC)\rightarrow \Delta_!(\B)$ factors as
\[
\xymatrix{\Delta_!(\INDZC)\ar[d]\ar[r]^{\;\;\;\;\;\;\;\;\;\;\overline{F}}&\Delta_!(\B)\\
\Delta_!(\Omega(\ZZ)^c)\ar[ur]_{F}&}.
\]
By Kashiwara we obtain the desired map $F:\Omega(\ZZ)^c\rightarrow \B$, and this concludes the proof of Theorem \ref{main2}.\\
\end{proof}

\end{document}